\documentclass[11pt]{amsart}
\usepackage[active]{srcltx}
\usepackage{verbatim}
\usepackage{epsfig,graphicx,color,mathrsfs}
\usepackage{graphicx}
\usepackage{amsmath,amssymb,amsthm,amsfonts}
\usepackage{amssymb}
\usepackage[english]{babel}

\usepackage[left=2.7cm,right=2.7cm,top=3cm,bottom=3cm]{geometry}

\newcommand{\R}{{\mathbb R}}

\newcommand{\teta }{\theta }

\numberwithin{equation}{section}
\newtheorem{theorem}{Theorem}[section]
\newtheorem{proposition}[theorem]{Proposition}
\newtheorem{lemma}[theorem]{Lemma}

\newtheorem{definition}[theorem]{Definition}
\newtheorem{remark}[theorem]{Remark}
\theoremstyle{definition}

\newcommand{\dyle}{\displaystyle}
\renewcommand{\dfrac}{\displaystyle\frac}
\newcommand{\brm}{\begin{remark}\rm}
\newcommand{\erm}{\end{remark}}
\newcommand{\brms}{\begin{remark}\rm}
\newcommand{\erms}{\end{remark}}
\newcommand{\bte}{\begin{theorem}}
\newcommand{\ete}{\end{theorem}}
\newcommand{\bpr}{\begin{proposition}}
\newcommand{\epr}{\end{proposition}}
\newcommand{\ble}{\begin{lemma}}
\newcommand{\ele}{\end{lemma}}
\newcommand{\beq}{\begin{equation}}
\newcommand{\eeq}{\end{equation}}
\newcommand{\bdm}{\begin{displaymath}}
\newcommand{\edm}{\end{displaymath}}
\numberwithin{equation}{section}

\newcommand{\bos}{\begin{remark}\rm}
\newcommand{\eos}{\end{remark}}

\newcommand{\ben}{\begin{enumerate}}
\newcommand{\een}{\end{enumerate}}

\newcommand{\g }{\gamma}

\newcommand{\be}{\begin{equation}}
\newcommand{\ee}{\end{equation}}

\title[Qualitative properties of positive solutions to nonlocal
critical problems ]{Qualitative properties of positive solutions to nonlocal
critical problems involving the Hardy-Leray potential}

\author[S. Dipierro, L. Montoro, I. Peral, B. Sciunzi]{Serena Dipierro, Luigi Montoro, Ireneo Peral, Berardino Sciunzi}
\address{S. Dipierro, Maxwell Institute
for Mathematical Sciences and School of Mathematics,
University of Edinburgh, James Clerk Maxwell Building,
Peter Guthrie Tait Road, Edinburgh EH9 3FD, United Kingdom}
\email{serena.dipierro@ed.ac.uk}

\address{L. Montoro, Dipartimento di Matematica, UNICAL, Ponte Pietro
Bucci 31 B, 87036 Arcavacata di Rende, Cosenza, Italy.}
\email{montoro@mat.unical.it}
\address{I. Peral, Departamento de Matem\'{a}ticas, Universidad Aut\'{o}noma de Madrid,
28049 Madrid, Spain.}
\email{ireneo.peral@uam.es}
\address{B. Sciunzi, Dipartimento di Matematica, UNICAL, Ponte Pietro
Bucci 31 B, 87036 Arcavacata di Rende, Cosenza, Italy.}
\email{sciunzi@mat.unical.it}

\begin{document}
\begin{abstract}
We prove  the existence, qualitative properties and asymptotic behavior of positive  solutions to the doubly critical problem
\begin{equation}\nonumber
(-\Delta)^s u=\vartheta\frac{u}{|x|^{2s}}+u^{2_s^*-1}, \quad u\in \dot{H}^s(\mathbb{R}^N).
\end{equation}
The technique that we use to prove the existence is based on variational arguments. The qualitative properties are obtained by using of the moving plane method, in a nonlocal setting, on the whole $\mathbb{R}^N$ and by some comparison results.

Moreover, in order to find the asymptotic behavior of solutions,
we use a representation result that allows to transform the original problem
into a different nonlocal problem in a weighted fractional space.
\end{abstract}

\maketitle
\medskip
\section{Introduction}\label{introdue}
In this work we consider the  doubly critical equation
\begin{equation}\label{Eq:P}
(-\Delta)^s u=\vartheta\frac{u}{|x|^{2s}}+u^{2_s^*-1}, \quad u\in \dot{H}^s(\mathbb{R}^N),
\end{equation}
with $N>2s$, $0<s<1$, $2^*_s:=\frac{2N}{N-2s}$ and
$0<\vartheta< \Lambda_{N,s}$, where $\Lambda_{N,s}$ is
the sharp constant of the Hardy-Sobolev inequality, that is
\begin{equation}\label{Hardy}
\Lambda_{N,s}\,\int_{\mathbb{R}^N} \frac{u^2(x)}{|x|^{2s}}\,dx
\le
\int_{\mathbb{R}^N} \,|\xi|^{2s} |\mathcal{F}(u)|^2\,d\xi,\quad {\mbox{ for any }} u\in\mathcal{C}^{\infty}_{0}(\mathbb{R}^N),\end{equation}
where $\mathcal{F}(u)$ denotes the Fourier transform of $u$.
Moreover,
$$ \Lambda_{N,s}= 2^{2s}\dfrac{\Gamma^2(\frac{N+2s}{4})}{\Gamma^2(\frac{N-2s}{4})},$$
where $\Gamma$ is the so-called Gamma function.
See \cite{He}, \cite{B}, \cite{FLS} and \cite{Y}.

We denote by $\mathcal S(\mathbb{R}^N)$ the class of all Schwartz functions in $\mathbb R^N$. Also, for any $f\in \mathcal S(\mathbb{R}^N)$, the fractional
Laplacian of $f$ will be denoted by~$(-\Delta)^s f$,
with~$s\in(0,1)$. Namely, for any~$x\in\R^N$,
\begin{equation*}
(-\Delta)^s f:=c_{N,s}\,P.V.\,\int_{\mathbb{R}^N}\frac{f(x)-f(y)}{|x-y|^{N+2s}}\,dy,
\end{equation*}
where the constant $c_{N,s}$ is given by
\begin{equation}\label{eq:cns}
c_{N,s}:=\left(\int_{\mathbb{R}^N}\frac{1-\cos(\xi_1)}{|\xi|^{N+2s}}\,d\xi\right)^{-1}=2^{2s-1}\pi^{-\frac{N}{2}}\frac{\Gamma(\frac{N+2s}{2})}{|\Gamma(-s)|},\end{equation}
see \cite{claudia, DPV12}.
\smallskip

Problem \eqref{Eq:P} can be considered as \textit{doubly critical} due to the critical power in the semilinear term and the spectral anomaly of the Hardy potential. The fractional framework introduces nontrivial difficulties that have interest in itself.
In this paper we want to show first the  existence of solutions to problem \eqref{Eq:P} and then both qualitative properties of solutions  and an asymptotic analysis of solutions   at zero and at infinity.  

First of all we need to define the natural functional framework for our problem, i.e. we consider the following Hilbert space:
\begin{definition}
Let $0<s<1$. We define the homogeneous fractional Sobolev space of order $s$ as
$$\dot H^s(\mathbb{R}^N):=\{u\in L^{2^{*}}(\mathbb{R}^N)\, :\, |\xi|^s\mathcal{F}(u)(\xi)\in L^2(\mathbb{R}^N)\},$$
namely the completion of $C^{\infty}_0(\mathbb{R}^N)$ under the norm
\begin{equation}\label{eq:fracnorm}
\|u\|^2_{\dot H^s}:=\int_{\mathbb{R}^N}|\xi|^{2s}|\mathcal{F}(u)(\xi)|^2\,d\xi.
\end{equation}
\end{definition}

By Plancherel's identity, we obtain an equivalent expression of the norm \eqref{eq:fracnorm}, as the following  result states
\begin{proposition}\label{antifou}
Let $N\geq 1$ and $0<s<1$. Then for all $u\in \dot H^s(\mathbb{R}^N)$
\begin{equation}\nonumber
 \int_{\mathbb{R}^N}|\xi|^{2s}|\mathcal{F}(u)(\xi)
|^2\, d\xi=c_{N,s}\int_{\mathbb{R}^N}\int_{\mathbb{R}^N}\frac{|u(x)-u(y)|^2}{|x-y|^{N+2s}}\,dx\,dy,\end{equation}
where $c_{N,s}$ is defined in \eqref{eq:cns}.
\end{proposition}
See for instance \cite{FLS}.
\begin{remark}
According to Proposition \ref{antifou} and using a density argument, the inequality in~\eqref{Hardy} can be reformulated in the following way:
\begin{equation}\label{hardy-1}
\Lambda_{N,s}\,\int_{\mathbb{R}^N}  \frac{u^2(x)}{|x|^{2s}}\,dx\le c_{N,s}\int_{\mathbb{R}^N}\int_{\mathbb{R}^N}\frac{|u(x)-u(y)|^2}{|x-y|^{N+2s}}\,dx\,dy, \quad {\mbox{  for any }} u \in  \dot H^s(\mathbb{R}^N).
\end{equation}
\end{remark}

The notion of solutions to~\eqref{Eq:P} that we consider in this paper
is given in the following definition:
\begin{definition}\label{def:sol}
We say that $u\in \dot H^s(\mathbb{R}^N)$,  is a weak solution to \eqref{Eq:P} if, for every $\varphi \in \dot H^s(\mathbb{R}^N)$, we have:
\begin{equation}\label{1.1bis}
\frac 12 c_{N,s}\int_{\mathbb{R}^N}\int_{\mathbb{R}^N}\frac{(u(x)-u(y))
(\varphi(x)-\varphi(y))}{|x-y|^{N+2s}}\,dx\,dy=\vartheta\int_{\mathbb{R}^N}\frac{u}{|x|^{2s}}\varphi \,dx
+\int_{\mathbb{R}^N}u^{2_s^*-1}\varphi\,dx.
\end{equation}
\end{definition}

In the local case the problem
\begin{equation}\label{eq:problterracini}
-\Delta u=A\frac{u}{|x|^{2}}+u^{2^*-1}  \quad\text{in}\,\, \mathbb{R}^N\setminus{\{0\}},
\end{equation}
with  $A\in [0, (N-2)^2/4)$ and~$2^*:=\frac{2N}{N-2}$, was studied in \cite{Te}. The author proves existence, uniqueness and qualitative properties of the solutions of problem \eqref{eq:problterracini} by a clever  use of variational arguments and of the moving plane method.
In particular it has been showed in~\cite{Te} that the solution of  \eqref{eq:problterracini} is unique (up to rescaling) and is given by
\begin{equation}\label{eq:classificaionterracini}
u_A(x)=\frac{(N(N-2)\eta_A^2)^{(N-2)/4}}{(|x|^{1-\eta_A}(1+|x|^{2\eta_A}))^{(N-2)/2}},
\end{equation}
where
$$\eta_A:=\left(1-\frac{4A}{(N-2)^2}\right)^{1/2}.$$

In the nonlocal setting, in~\cite{CLO} the authors study the nonlocal problem \begin{equation}\label{eq:chenliou}
(-\Delta)^s u=u^{2_s^*-1} \quad\text{in}\,\, \mathbb{R}^N.
\end{equation}
They prove that every  positive solution $u$ of~\eqref{eq:chenliou} is radially symmetric and radially  decreasing about some point $x_0\in \mathbb{R}^N$ and is given by \begin{equation}\nonumber
u(x)=c\left(\frac{t}{t^2+|x-x_0|^2}\right)^{(N-2s)/2},
\end{equation}
where $c$ and $t$ are positive constants. The authors use the moving plane method in an integral form and then they classify the solutions  using that in fact problem \eqref{eq:chenliou} is equivalent to the integral equation
\begin{equation}\label{eq:integralchenliou}
u(x)=\int_{\mathbb{R}^N}\frac{1}{|x-y|^{N-2s}}u(y)^{\frac{N+2s}{N-2s}}\,dy.
\end{equation}
\smallskip

Here we first show that problem~\eqref{Eq:P} has a positive solution,
according to Definition~\ref{def:sol}. Namely, we have the following:
\begin{theorem}\label{th:exsolution}
Let~$0\le\vartheta< \Lambda_{N,s}$. Then problem~\eqref{Eq:P} has a positive solution.
\end{theorem}
In order to prove Theorem~\ref{th:exsolution}, we deal with a
constrained maximization problem (see in particular the forthcoming
formula~\eqref{eq:max}) and then we use  the Lagrange multipliers technique
to get a solution to~\eqref{Eq:P} and the main difficulty is to proof the compactness.
Moreover, by local estimates we have that  the  positive solutions to \eqref{Eq:P}, are strong solutions in $\mathbb{R}^N\setminus \{0\}$.

\smallskip

The second result that we prove in this paper is the radial symmetry of
every solution to~\eqref{Eq:P}. That is, we have the following:

\begin{theorem}\label{thm:radiality}
Assume that $0\le \vartheta< \Lambda_{N,s}$ and let $u$ be a positive  solution to  \eqref{Eq:P} (in the sense of the Definition \ref{def:sol}). Then $u$ is radial and radially decreasing with respect to the origin. Namely there exists some strictly decreasing function $v:(0,+\infty)\rightarrow (0,+\infty)$ such that
$$u(x)=v(r), \quad r=|x|.$$
\end{theorem}
In the local case such a result is generally proved exploiting the moving plane method  which goes back to the seminal works of Alexandrov \cite{A} and Serrin  \cite{S}. See in particular  the celebrated papers \cite{BN, GNN}. In the non local case we refer to
\cite{BMS, CLO, FeWa, JW2, JW3}.

Moreover, the presence of the Hardy potential in equation \eqref{Eq:P} makes difficult
to use the technique
developed in~\cite{CLO} where the authors exploit  the equivalence of \eqref{Eq:P} to an integral equation like~\eqref{eq:integralchenliou}.

For this reason, in order to prove the radial symmetry of every solution to \eqref{Eq:P}, we  use an  approach based on the moving plane method  in all $\mathbb{R}^N$  for weak solutions of the equation.
\smallskip

Finally, we deal with the asymptotic behavior of solutions to~\eqref{Eq:P}
near the origin and at infinity.
For this, we use a representation result by by Frank, Lieb and Seiringer, see \cite{FLS}. In this way we are able to work with  an equivalent equation  of \eqref{Eq:P}, see  formula~\eqref{eq:equivalent} and then define a new nonlocal problem in some weighted fractional space. A similar procedure has been developed  in \cite{AMPP} and \cite{AMPP1} in order to solve some different elliptic and parabolic problems.

In particular, we first provide some regularity results
(see Subsections~\ref{sec:infi} and~\ref{sec:harn}).
Then, the asymptotic analysis is contained in the following:

\begin{theorem}\label{thm:main}
Let $u\in\dot{H}^s(\R^N)$ be a solution to \eqref{Eq:P}.
Then there exist two positive constants $c$ and $C$ such that
\begin{equation}\label{eq:sdafdsfa}
\frac{c}{\Big(|x|^{1-\eta_{\theta}}(1+|x|^{2\eta_{\theta}})\Big)^{\frac{N-2s}{2}}}\leq u(x)\leq \frac{C}{\Big(|x|^{1-\eta_{\theta}}(1+|x|^{2\eta_{\theta}})\Big)^{\frac{N-2s}{2}}}, \quad \text{in}\,\,\mathbb{R}^N\setminus \{0\},
\end{equation}
where
\begin{equation}\label{4.16bis}
\eta_{\theta}= 1-\frac{2\alpha_{\theta}}{N-2s}
\end{equation}
and $\alpha_{\theta}\in (0, (N-2s)/2)$ is
a suitable parameter whose explicit value will be determined as the
unique  solution to  equation \eqref{eq:alpha}.
\end{theorem}

We point out the similarities between formulas~\eqref{eq:classificaionterracini} and~\eqref{eq:sdafdsfa}. Indeed, the parameter~$\eta_{\theta}$ in Theorem~\ref{thm:main} plays the same role as the parameter~$\eta_A$
in the classical problem~\eqref{eq:problterracini}
(i.e. when~$s=1$ and the fractional Laplacian boils down to the classical Laplacian) both for the behavior near~0
and at infinity of the solution.

\

\noindent The paper is organized as follows: in Section \ref{sec:existence} we show the existence of at least a solution to~\eqref{Eq:P}  and we prove Theorem \ref{th:exsolution}. In Section \ref{sec:symmetry} we study the qualitative properties of solutions to~\eqref{Eq:P}. We provide some maximum/comparison principle and we perform the moving plane method in order to get the radial symmetry of the solutions and prove then Theorem \ref{thm:radiality}. Finally in Section \ref{sec:Asymptotic} we investigate the behavior of solutions to \eqref{Eq:P} and we prove Theorem~\ref{thm:main}.

\

\noindent {\bf Notation.} Generic fixed and numerical constants will be denoted by
$C$ (with subscript in some case) and they will be allowed to vary within a single line or formula.

\section{Existence: the maximization problem and proof of Theorem~\ref{th:exsolution}}\label{sec:existence}
To prove the existence of a solution to \eqref{Eq:P} we consider the following maximization problem
\begin{equation}\label{eq:max}
S(\vartheta):=\sup_{u\in \dot{H}^s(\mathbb{R}^N)\setminus \{0\}}Q(u),
\end{equation}
where
$$
Q(u):=
\dfrac{\displaystyle\int_{\mathbb{R}^N}|u|^{2^*_s}\,dx}{\displaystyle\left(\frac{c_{N,s}}{2}\int_{\mathbb{R}^N}\int_{\mathbb{R}^N}\frac{|u(x)-u(y)|^2}{|x-y|^{N+2s}}\,dx\,dy
-\vartheta \int_{\mathbb{R}^N}\frac{u^2}{|x|^{2s}}\,dx\right)^{\frac{2_s^*}{2}}}.
$$
Let  us define the continuous bilinear form $\mathcal L: \dot{H}^s(\mathbb{R}^N)\times \dot{H}^s(\mathbb{R}^N)\rightarrow \mathbb{R}$ as
\begin{equation}\label{eq:bilinear}
\mathcal L(u,v):=\frac{c_{N,s}}{2}\int_{\mathbb{R}^N}\int_{\mathbb{R}^N}\frac{(u(x)-u(y))(v(x)-v(y))}{|x-y|^{N+2s}}\,dx\,dy-\vartheta\int_{\mathbb{R}^N}\frac{uv}{|x|^{2s}}\,dx
\end{equation}
and the quadratic form $\tilde{\mathcal L}: \dot{H}^s(\mathbb{R}^N)\rightarrow \mathbb{R}$ as
\begin{equation}\label{eq:quadraticccc}
\tilde{\mathcal L}(u):=\mathcal L(u,u)=\frac{c_{N,s}}{2}\int_{\mathbb{R}^N}\int_{\mathbb{R}^N}\frac{|u(x)-u(y)|^2}{|x-y|^{N+2s}}\,dx\,dy-\vartheta\int_{\mathbb{R}^N}\frac{u^2}{|x|^{2s}}\,dx.
\end{equation}
We point out that, for $0\le \vartheta < \Lambda_{N,s}$,
a direct application of Hardy inequality (see~\eqref{hardy-1}) shows
that $(\tilde{\mathcal L}(u))^{1/2}$ is an equivalent norm in $\dot{H}(\mathbb{R}^N)$. We readily note that, by Hardy inequality,   we have that $S(\vartheta)<+\infty$.

Moreover it easy to see that, if $u$ is a maximum of the problem~\eqref{eq:max}, then all the rescaled functions of $u$ of the form
\begin{equation}\label{eq:scalinginvariance}
\sigma^{-\frac{N-2s}{2}}u\Big(\frac{\cdot}{\sigma}\Big),\quad \sigma >0
\end{equation}
are also solutions to the maximization problem \eqref{eq:max}.

To get the existence, we take advantage of some improved Sobolev inequalities, see \cite[Theorem 1.1]{PP}. In particular, in the proof  of Theorem \ref{th:exsolution} we will use the fact that, for a function
$u\in \dot H^s(\R^N)$, we have that
\begin{equation}\label{eq:Morrey}
\|u\|_{L^{2_s^*}}\leq C\|u\|^{\theta}_{\dot H^s}\|u\|^{1-\theta}_{\mathcal L^{2,N-2s}},
\end{equation}
where $2/2_s^*\leq \teta<1$ and $\|\cdot\|_{\mathcal L^{2,N-2s}}$ denotes the norm in the Morrey space $\mathcal L^{2,N-2s}$, that is
\begin{equation}\label{eq:MorreyNorm}
\|u\|^2_{\mathcal L^{2,N-2s}}:=\underset{R>0;\, x\in \mathbb{R}^N}{\sup}\frac{R^{N-2s}}{|B_R(x)|}\int_{B_R(x)}|u|^2\,dz.
\end{equation}

\begin{proof}[Proof of Theorem \ref{th:exsolution}.]  We start finding a maximizing sequence $\{v_n\}$ that consist of radial and radially decreasing  functions,
i.e. $v_n(x)=v_n(|x|)$ for any~$x\in\R^N$.
In fact, let us first consider a maximizing  sequence $\{u_n\}\in \dot{H}^s(\mathbb{R}^N)$ for  \eqref{eq:max}. Notice that it is not restrictive to assume that  $u_n(x)\geq0$ a.e. in $\mathbb{R}^N$ (if not take $|u_n(x)|$).

Define $v_n(x):=u_n^{*}(x)$, where by $f^*$ we denote the decreasing rearrangement of a measurable function $f$ (where $f$ is such that  all its positive level set have finite measure), namely
$$f^*(x)=\inf \{ t>0\, :\, |\{y\in \mathbb{R}^N\,:\, u(y)>t\}|\leq \omega_N|x|^N \},$$
where $\omega_N$ is the volume of the standard unit $N$-sphere. By using a Polya-Szeg\"o type  inequality, see~\cite{YJP}, we have that
\begin{equation}\nonumber
\frac{c_{N,s}}{2}\int_{\mathbb{R}^N}\int_{\mathbb{R}^N}\frac{|u(x)-u(y)|^2}{|x-y|^{N+2s}}\,dx\,dy\geq\frac{c_{N,s}}{2}\int_{\mathbb{R}^N}\int_{\mathbb{R}^N}\frac{|u^*(x)-u^*(y)|^2}{|x-y|^{N+2s}}\,dx\,dy
\end{equation}
and, by rearrangement properties, we also have  that
\begin{equation}\nonumber
\int_{\mathbb{R}^N}|u|^{2^*_s}\, dx=\int_{\mathbb{R}^N}|u^*|^{2^*_s}\, dx,\quad \int_{\mathbb{R}^N}\frac{u^2}{|x|^2}\, dx\leq \int_{\mathbb{R}^N}\frac{(u^*)^2}{|x|^2}\, dx.
\end{equation}
Then the sequence $\{v_n\}$ is also  a  maximizing sequence (of radial and  decreasing functions)  for~\eqref{eq:max}.
Thanks to  the homogeneity of \eqref{eq:quadraticccc}, we do assume
$\tilde{\mathcal L}(v_n)=1$  for all $n$ and
\begin{equation}\label{eq:shjaghsdanuei}
\int_{\mathbb{R}^N}|v_n|^{2^*_s}\,dx \rightarrow S(\vartheta)>0,\quad\text{as }\,\,  n\rightarrow +\infty.
\end{equation}

Now we are going to show that  a suitable rescaled sequence, that we  will denote by $\hat v_n$,  converges to a nontrivial  weak limit, that is
\begin{equation}\label{2.7bis}
\hat v_n \rightharpoonup v \not\equiv 0 \quad \text{in } \dot{H}^s(\mathbb{R}^N)
\quad {\mbox{ as }} n\to+\infty.
\end{equation}
To prove this, we first observe that, from \eqref{eq:Morrey} and \eqref{eq:shjaghsdanuei}, we have that
$$\|v_n\|_{\mathcal L^{2,N-2s}}\geq C>0,
$$
for some~$C$ independent of~$n$.

By \eqref{eq:MorreyNorm} and the fact that $\tilde{\mathcal L}(v_n)=1$, for any $n\in \mathbb{N}$, we get the existence of $R_n>0$ and $x_n\in\R^N$ such that
\begin{equation}\label{eq:morreyconsequence}
\frac{1}{R^{2s}_n}\int_{B_{R_n}(x_n)}|v_n(z)|^2\, dz\geq C>0,
\end{equation}
for some new positive constant $C$ that does not depend on $n$.

Now we define  the sequence  $\hat v_n$ (of  symmetric,   radial  decreasing functions) as
$$\hat v_n(x):=R_n^{\frac{N-2s}{2}}v_n(R_nx).$$
Notice that, by~\eqref{eq:scalinginvariance},
\begin{equation}\label{starrrr}
\int_{\mathbb{R}^N}|\hat v_n|^{2^*_s}\,dx =
\int_{\mathbb{R}^N}|v_n|^{2^*_s}\,dx \rightarrow S(\vartheta)>0,\quad\text{as }\,\,  n\rightarrow +\infty.
\end{equation}
Moreover, using again the scaling invariance \eqref{eq:scalinginvariance},
we still have
\begin{equation}\label{usare}
\tilde{\mathcal L}(\hat v_n)=1\end{equation}
and
$$\|\hat v_n\|_{\dot H^s}\leq C.$$
Then, there exists~$v\in \dot H^s(\mathbb{R}^N)$ such that, up to subsequences, $\hat v_n\rightharpoonup v$ in $\dot H^s(\mathbb{R}^N)$ as~$n\to+\infty$.

Hence, to finish the proof of~\eqref{2.7bis}, it remains to show that
\begin{equation}\label{2.7ter}
v \not\equiv 0.
\end{equation}
To do  this, we change variable in~\eqref{eq:morreyconsequence} and we obtain that
\begin{equation}\label{eq:morreyconsequence1}
\int_{B_{1}(\hat x_n)}|\hat v_n(x)|^2\, dx\geq C>0,
\end{equation}
where~$\hat x_n:=x_n/R_n$. Now we deal with two cases separately.

\begin{itemize}
\item[$(i)$] Let us suppose that, up to subsequence,  the sequence of points $\hat x_n \rightarrow \infty$.
From \eqref{eq:morreyconsequence1} we infer that for every $n$ there exists a set $A_n$ of positive Lebesgue  measure,  such that  $\hat v_n(x)>0$, if $x\in A_n$. Since the sequence  $\hat v_n$  consists of radial and radially decreasing functions, we have for all $n$ that
\begin{equation*}
\int_{B_{2}(0)}|\hat v_n(x)|^2\, dx
\geq\int_{B_{1}(0)}|\hat v_n(x+\hat x_n)|^2\, dx=\int_{B_{1}(\hat x_n)}|\hat v_n(x)|^2\, dx\geq C>0.
\end{equation*}
Therefore, since the embedding of $\dot{H}^s(\mathbb{R}^N)\hookrightarrow L^p_{loc},\,1\leq p<2_s^*$ is compact, see ~\cite[Corollary~$7.2$]{DPV12}, we deduce that
$$\int_{B_{2}(0)}|v(x)|^2\geq C$$
and thus $v\not\equiv0$. This shows~\eqref{2.7ter} in this case.
\item[(ii)] Let us suppose   up to subsequence  $\hat x_n \rightarrow x_0$. Then, in this case, we fix a compact set~$\mathcal K$ sufficiently large such that $B_2(x_0)\subset \mathcal K$. Making again use of
\eqref{eq:morreyconsequence1}, we have that, for $n$ large enough,
$$ \int_{\mathcal K}|\hat v_n(x)|^2\, dx
\geq\int_{B_{1}(\hat x_n)}|\hat v_n(x)|^2\, dx \geq C>0.
$$
Therefore, the $L^2$ strong convergence on $\mathcal K$ implies that
$$ \int_{\mathcal K}|v(x)|^2\geq C$$
and then
 $v\not\equiv0$. This concludes the proof of~\eqref{2.7ter} also in this case.
\end{itemize}
Having finished the proof of~\eqref{2.7ter}, we obtain~\eqref{2.7bis}.

Now, since $\{\hat v_n\}$ is a maximizing sequence, we can show that actually
\begin{equation}\label{2.8bisbis}
{\mbox{$\hat v_n \rightarrow v$ strongly in $ \dot{H}^s(\mathbb{R}^N)$.}}
\end{equation}
In order to prove this, we observe that, recalling~\eqref{eq:bilinear}
and~\eqref{eq:quadraticccc},
\begin{equation}\label{eq:bilinear1}
\tilde{\mathcal L}(v)+\tilde{\mathcal L}(\hat v_n-v)=1+o(1),
\end{equation}
where~$o(1)$ denotes a quantity that tends to zero as~$n\to+\infty$.
Indeed, we have that
$$\tilde{\mathcal L}(\hat v_n)=\tilde{\mathcal L}(v+ \hat v_n- v)=\tilde{\mathcal L}(v)+\tilde{\mathcal L}(\hat v_n-v)+2\mathcal{L}(v,\hat v_n- v).$$
Moreover, by~\eqref{2.7bis}
$$\mathcal{L}(v,\hat v_n-v)\rightarrow 0\quad  {\mbox{ as }}n\to+\infty.$$
The last two formulas and~\eqref{usare} imply~\eqref{eq:bilinear1}.

Furthermore, since $\hat v_n\rightarrow v$ a.e. (due to the compact embedding $\dot{H}^s(\mathbb{R}^N)\hookrightarrow L^p_{loc},\, 1\leq p<2_s^*$,
see~\cite[Corollary $7.2$]{DPV12}), by~\eqref{starrrr} and
Brezis-Lieb result \cite{BL} we have the following
\begin{eqnarray}\nonumber
S(\vartheta)=\lim_{n\rightarrow+\infty}\int_{\mathbb{R}^N}|\hat v_n|^{2^*_s}\,dx&=&\lim_{n\rightarrow+\infty}\left(\int_{\mathbb{R}^N}|v|^{2^*_s}\,dx+ \int_{\mathbb{R}^N}|\hat v_n-v|^{2^*_s}\,dx\right)\\\nonumber
&\leq&S(\vartheta)\Big(\tilde{\mathcal{L}}(v)\Big)^{\frac{2^*_s}{2}}+S(\vartheta)\Big(\underset{n\rightarrow +\infty}{\lim}\,\tilde{\mathcal{L}}(\hat v_n-v)\Big)^{\frac{2^*_s}{2}}\\\nonumber
&\leq& S(\vartheta)\Big(\tilde{\mathcal{L}}(v)+\underset{n\rightarrow +\infty}{\lim}\,\tilde{\mathcal{L}}(\hat v_n-v)\Big)^{\frac{2^*_s}{2}}\leq S(\vartheta),
\end{eqnarray}
where we used
\eqref{eq:bilinear1} in the last line. Therefore, all the inequalities above
have to be equalities. Moreover, since $v\not\equiv0$, we infer that $\tilde{\mathcal{L}}(v)=1$ and~$\underset{n\rightarrow +\infty}{\lim}\,\tilde{\mathcal{L}}(\hat v_n-v)=0$, i.e. $\hat v_n \rightarrow v$ strongly in $ \dot{H}^s(\mathbb{R}^N)$ (recall that~$(\tilde{\mathcal L}(\cdot))^{1/2}$ is an equivalent norm
in~$\dot{H}^s(\mathbb{R}^N)$). This shows~\eqref{2.8bisbis}.

As a consequence of~\eqref{2.8bisbis} and using the fractional
Sobolev embedding,
we have that~$\hat v_n \rightarrow v$ strongly in~$L^{2^*_s}(\mathbb{R}^N)$ as well. Also, $ v$ turns to be a maximum for \eqref{eq:max}.

It is now standard by Lagrange multiplier Theorem to get the existence of a solution to \eqref{Eq:P}, and so the proof of Theorem~\ref{th:exsolution} is
finished.
\end{proof}

\section{Symmetry of Solutions and proof of Theorem~\ref{thm:radiality}}\label{sec:symmetry}

In this section we show that all the solution of \eqref{Eq:P} are radial  and radially decreasing with respect to the origin, as stated in Theorem~\ref{thm:radiality}.

\subsection{Comparison principles}
In this subsection, relying on some results of Silvestre (see Section 2 in \cite{Si}),
we provide maximum/comparison principle that will be useful in the application of the moving plane method in the forthcoming Subection~\ref{subec rad}.

For this, we introduce some notations. Let
\begin{equation}\label{fundamental}
\Phi (x)\,:=\, \frac{C}{|x|^{N-2s}}
\end{equation}
be the fundamental solution of $(-\Delta)^s$. We denote by~$\Gamma$ the regularization of~$\Phi$ constructed in~\cite{Si} (see Figure 2.1 there) and we set, for any~$\tau>1$ and for any~$x\in\mathbb{R}^N$,
\begin{equation}\label{reggamma}
\Gamma_\tau (x)\,:=\,\frac{\Gamma \left(\frac{x}{\tau}\right)}{\tau^{N-2s}}
\end{equation}
and
\begin{equation}\label{reggammabis}
\gamma_\tau(x) \,:=(-\Delta)^s\Gamma_\tau(x).
\end{equation}
The function~$\Gamma_\tau\in C^{1,1}(\mathbb{R}^N)$ coincides with~$\Phi$ outside~$B_\tau(0)$ and it is a paraboloid inside~$B_\tau(0)$ (see Section~$2.2$
in~\cite{Si}), that is
\begin{eqnarray}
&& \Gamma_\tau(x)=\Phi(x) {\mbox{ in $\mathbb{R}^N\setminus B_\tau(0)$}}\label{fuori}\\
{\mbox{and }}&& \Gamma_\tau(x)\le\Phi(x) {\mbox{ in $B_\tau(0)$.}}\label{dentro}
\end{eqnarray}
Moreover, the function~$\gamma_\tau$ is strictly positive, thanks to Proposition~$2.12$ in~\cite{Si}.

Furthermore, given a function~$\omega\in\dot{H}^s(\mathbb{R}^N)\cap C(\overline\Omega)$,
we say that~$\omega$ satisfies
\[
(-\Delta)^s\omega\geq 0\qquad\qquad\text{in}\quad\Omega,
\]
if
\begin{equation}\nonumber
\frac{c_{N,s}}{2}\int_{\mathbb{R}^N}\int_{\mathbb{R}^N}\frac{(\omega(x)-\omega(y))(\varphi(x)-\varphi(y))}
{|x-y|^{N+2s}}\,dx\,dy\geq 0,
\end{equation}
for any nonnegative test function $\varphi$ with compact support in $\Omega$.

We denote a point $x\in\mathbb{R}^N$ as $x=(x_1,x_2,\cdots,x_N)$ and, for  any~$\lambda \in \mathbb{R}$, we set
\begin{equation}\label{sigma l}
\Sigma_{\lambda}:=\Big\{x\in \mathbb{R}^N : x_1< \lambda\Big\}
\end{equation}
and
\begin{equation}\label{T l}
T_{\lambda}:=\Big\{x\in \mathbb{R}^N : x_1=\lambda\Big\}.
\end{equation}
For any~$\lambda \in \mathbb{R}$, we also set
\begin{equation}\label{x lambda}
x_\lambda:=(2\lambda-x_1,x_2,\cdots,x_N).
\end{equation}

With these definitions we can prove the following maximum principle:

\begin{proposition}\label{gfjhbgfjhfjhjnbncbncbcnvakodkod}
Let $\lambda\in\mathbb{R}$ and $\Omega\Subset\Sigma_{\lambda}$ be a bounded open set. Let also~$\omega\in\dot{H}^s(\mathbb{R}^N)\cap C(\overline\Omega)$ that satisfies
\[
(-\Delta)^s\omega\geq 0\qquad\qquad\text{in}\quad\Omega\,.
\]
Assume that $\omega$ is nonnegative in $\Sigma_{\lambda}$ and odd with respect to the hyperplane $T_\lambda$.

Then, either $\omega\equiv0$ in~$\mathbb{R}^N$ or~$\omega>0$ in~$\Omega$.
\end{proposition}

\begin{remark}
The proof that we are going to exploit is the one in \cite{Si} (see Proposition 2.17). Some changes are needed since in our case we do not assume that $\omega$ is nonnegative in the whole space but, for future use, we assume that $\omega$ is odd with respect to the hyperplane $T_\lambda$. We could say that, in some sense, we agree that $\omega$ can have a negative part, but the latter has to be not to large.
\end{remark}

\begin{proof}[Proof of Proposition~\ref{gfjhbgfjhfjhjnbncbncbcnvakodkod}]
If~$\omega>0$ in~$\Omega$, then the claim is true.
Therefore, suppose that there exists a
point~$x_0\in\Omega$ such that~$\omega(x_0)=0$.
Hence, by \cite[Proposition 2.15]{Si}, we have that,
for $\tau<\text{dist}(x_0,\partial\Omega)$,
\begin{equation}\label{minore}
0=\omega(x_0)\geq \int_{\mathbb{R}^N}\,\omega(x)\,\gamma_\tau(x-x_0)\,dx\,,
\end{equation}
where~$\gamma_\tau$ is defined in~\eqref{reggammabis}.

We claim that, for $\tau<\text{dist}(x_0,\partial\Omega)$,
\begin{equation}\label{jdbkvkjxvnkxvbvbvzxzx}
 \int_{\mathbb{R}^N}\,\omega(x)\,\gamma_\tau(x-x_0)\,dx\geq 0\,.
\end{equation}
For this, we notice that
\begin{equation}\label{forse}
\gamma_\tau(x-x_0)\ge\gamma_\tau(x_\lambda-x_0)>0 \quad {\mbox{ for any }}x\in\Sigma_\lambda.
\end{equation}
Indeed, if~$x\in\Sigma_\lambda\setminus B_\tau(x_0)$, then
\begin{equation}\label{hghghfghgfgfdkkddk}
\begin{split}
\gamma_\tau(x-x_0)&=\int_{\mathbb{R}^N}\,\frac{\Gamma_\tau(x-x_0)-\Gamma_\tau(y)}{|x-x_0-y|^{N+2s}}\,dy\\
&=\int_{\mathbb{R}^N}\,\frac{\Phi(x-x_0)-\Phi(y)}{|x-x_0-y|^{N+2s}}\,dy\,
+\,\int_{\mathbb{R}^N}\,\frac{\Phi(y)-\Gamma_\tau(y)}{|x-x_0-y|^{N+2s}}\,dy\\
&=\int_{B_\tau(x_0)}\,\frac{\Phi(y-x_0)-\Gamma_\tau(y-x_0)}{|x-y|^{N+2s}}\,dy\,,
\end{split}
\end{equation}
where in the last step we have used the fact that~$\Phi$ is the fundamental
solution to~$(-\Delta)^s$ and~\eqref{fuori}
(all the integrals have to be intended in the principal value sense).
Similarly, one has (again in the principal value sense)
\begin{equation}\label{gggggggggggggggg}
\gamma_\tau(x_\lambda-x_0)=\int_{B_\tau(x_0)}\,\frac{\Phi(y-x_0)-\Gamma_\tau(y-x_0)}{|x_\lambda-y|^{N+2s}}\,dy.
\end{equation}
Since~$|x-y|\le|x_\lambda-y|$ if~$x\in\Sigma_\lambda\setminus B_\tau(x_0)$
and~$y\in B_\tau(x_0)$, from~\eqref{hghghfghgfgfdkkddk} and~\eqref{gggggggggggggggg} we obtain that
\begin{equation}\label{dskghjkdbhvhbvhc}
\gamma_\tau(x-x_0)\ge\gamma_\tau(x_\lambda-x_0) {\mbox{ for any }}x\in\Sigma_\lambda\setminus B_\tau(x_0).
\end{equation}

Let now $ x\in B_\tau(x_0)$. In~\cite{Si} it has been proved that, for~$\tau$ small,
\begin{equation}\label{decad}
\gamma_\tau(x'-x_0)\leq \frac{c\,\tau^{2s}}{|x'-x_0|^{N+2s}}\quad \text{for}\,\, |x'-x_0|\geq\frac{\text{dist}(x_0,T_\lambda)}{2},
\end{equation}
where~$c>0$.
Notice that, if~$x\in B_\tau(x_0)$, then~\eqref{decad} holds for~$x':=x_\lambda$.
In particular, for~$\tau$ small,
$$ \gamma_\tau(x_\lambda-x_0)\le C,$$
for some positive constant~$C$.
Moreover
\[
\gamma_\tau(x-x_0)=\frac{1}{\tau^N}\gamma_1\left(\frac{x-x_0}{\tau}\right)\,.
\]
Now, we choose~$\tau$ sufficiently small such that
$$ \frac{1}{\tau^N}\gamma_1\left(\frac{x-x_0}{\tau}\right)\ge C, $$
and this implies that
\begin{equation}\label{dskghjkdbhvhbvhcgdfjsgfjgfjs}
\gamma_\tau(x-x_0)\geq \gamma_\tau(x_\lambda-x_0)\geq 0\qquad\text{for}\quad {\mbox {for any }} x\in B_\tau(x_0)\,.
\end{equation}
Putting together~\eqref{dskghjkdbhvhbvhc} and~\eqref{dskghjkdbhvhbvhcgdfjsgfjgfjs}
we obtain~\eqref{forse}.

Now, in order to prove~\eqref{jdbkvkjxvnkxvbvbvzxzx}, we write
$$ \int_{\mathbb{R}^N}\,\omega(x)\,\gamma_\tau(x-x_0)\,dx=\int_{\Sigma_\lambda}\,\omega(x)\,\gamma_\tau(x-x_0)\,dx+\int_{\mathbb{R}^N\setminus\Sigma_\lambda}\,\omega(x)\,\gamma_\tau(x-x_0)\,dx. $$
Therefore,~\eqref{jdbkvkjxvnkxvbvbvzxzx} is equivalent to show that
\begin{equation}\label{piu}
\int_{\Sigma_\lambda}\,\omega(x)\,\gamma_\tau(x-x_0)\,dx\ge -\int_{\mathbb{R}^N\setminus\Sigma_\lambda}\,\omega(x)\,\gamma_\tau(x-x_0)\,dx.
\end{equation}
For this, we recall that~$\omega\ge0$ in~$\Sigma_\lambda$ and it is odd with respect to~$T_\lambda$, and so, using~\eqref{forse} and the change of variable~$y=x_\lambda$, we have
\begin{eqnarray*}
\int_{\Sigma_\lambda}\,\omega(x)\,\gamma_\tau(x-x_0)\,dx &\ge & \int_{\Sigma_\lambda}\,\omega(x)\,\gamma_\tau(x_\lambda-x_0)\,dx\\
&=&-\int_{\Sigma_\lambda}\,\omega(x_\lambda)\,\gamma_\tau(x_\lambda-x_0)\,dx\\
&=& -\int_{\mathbb{R}^N\setminus\Sigma_\lambda}\,\omega(y)\,\gamma_\tau(y-x_0)\,dx.
\end{eqnarray*}
This implies~\eqref{piu}, and therefore~\eqref{jdbkvkjxvnkxvbvbvzxzx}.

As a consequence, from~\eqref{minore} and~\eqref{jdbkvkjxvnkxvbvbvzxzx}, we have
\begin{equation}\nonumber
\int_{\mathbb{R}^N}\,\omega(x)\,\gamma_\tau(x-x_0)\,dx\,=\, 0\,.
\end{equation}
Since $\gamma_\tau$ is strictly positive, this implies that~$\omega=0$ a.e. in~$\mathbb{R}^N$ and concludes the proof.
\end{proof}

\subsection{Radial symmetry  of the solutions}\label{subec rad}
In this section we prove that every solution $u\in \dot{H}^s(\mathbb{R}^N)$ to \eqref{Eq:P} is actually symmetric and monotone decreasing around the origin. The proof will be carried out exploiting the moving plane method. For the case of bounded domain in the nonlocal case we refer to \cite{BMS, FeWa, JW2, JW3}. To do this, we start considering without lost of generality the $x_1$-direction.

For any~$\lambda \in \mathbb{R}$, we recall the definitions of~$\Sigma_\lambda$
and~$T_\lambda$ given in~\eqref{sigma l} and~\eqref{T l}, respectively.
Moreover, we also set, for~$x=(x_1,\ldots,x_N)\in\mathbb{R}^N$ and~$x_\lambda$ defined in~\eqref{x lambda},
$$ u_{\lambda}(x):=u(x_\lambda).$$
A point in $\mathbb{R}^N\times\mathbb{R}^N$ is denoted by $(x,y)$ with $x,y\in\mathbb{R}^N$.

With these definitions, we have that, for any~$\varphi\in\dot{H}^s(\R^N)$,
\begin{eqnarray*}
&&\frac{c_{N,s}}{2}\int_{\mathbb{R}^N}\int_{\mathbb{R}^N}\frac{(u_\lambda(x)-u_\lambda(y))(\varphi(x)-\varphi(y))}{|x-y|^{N+2s}}\,dx\,dy\\
&=&\frac{c_{N,s}}{2}\int_{\mathbb{R}^N}\int_{\mathbb{R}^N}\frac{(u(x_\lambda)-u(y_\lambda))(\varphi(x)-\varphi(y))}{|x-y|^{N+2s}}\,dx\,dy\\
&=&\frac{c_{N,s}}{2}\int_{\mathbb{R}^N}\int_{\mathbb{R}^N}\frac{(u(t)-u(z))(\varphi(t_\lambda)-\varphi(z_\lambda))}{|t_\lambda-z_\lambda|^{N+2s}}\,dt\, dz\\
&=&\frac{c_{N,s}}{2}\int_{\mathbb{R}^N}\int_{\mathbb{R}^N}\frac{(u(t)-u(z))(\varphi(t_\lambda)-\varphi(z_\lambda))}{|t-z|^{N+2s}}\,dt\,dz\\
&=&\frac{c_{N,s}}{2}\int_{\mathbb{R}^N}\int_{\mathbb{R}^N}\frac{(u(t)-u(z))(\varphi_\lambda(t)-\varphi_\lambda(z))}{|t-z|^{N+2s}}\,dt\,dz\\
&=& \vartheta\int_{\mathbb{R}^N}\frac{u(t)}{|t|^{2s}}\,\varphi_\lambda(t)\,dt +\int_{\mathbb{R}^N}u^{2^*_s-1}(t)\varphi_\lambda(t)\,dt\\
&=&\vartheta\int_{\mathbb{R}^N}\frac{u_\lambda(x)}{|x_\lambda|^{2s}}\,\varphi(x)\,dx +\int_{\mathbb{R}^N}u_\lambda^{2^*_s-1}(x)\varphi(x)\,dx,
\end{eqnarray*}
where the changes of variables~$t=x_\lambda$ and~$z=y_\lambda$
and~\eqref{1.1bis} were used. Notice that,   if~$\varphi\in\dot{H}^s(\R^N)$,
then~$\varphi_\lambda\in\dot{H}^s(\R^N)$ and so~$\varphi_\lambda$ can be used as a test function in~\eqref{1.1bis}.

As a consequence, $u_\lambda$ is a weak solution to
\begin{equation}\label{hgdhgdhfgjssfgsf}
(-\Delta)^s u_\lambda=\vartheta\frac{u_\lambda}{|x_\lambda|^{2s}}+u_\lambda^{2_s^*-1} \quad\text{in}\,\, \mathbb{R}^N.
\end{equation}

Now we prove the following:
\begin{lemma}\label{hfksdkhfksfhkfhdjb}
Let~$0\leq \vartheta< \Lambda_{N,s}$ and let $u$ be a positive  solution to  \eqref{Eq:P}, in the sense of Definition~\ref{def:sol}. Then
\begin{equation}\label{eq:liminfinito}
\lim_{|x|\rightarrow 0}u(x)=+\infty.
\end{equation}
\end{lemma}

\begin{proof} By the weak Harnack inequality
we have that~$\displaystyle\inf_{B_2(0)} u>\delta>0$. In particular,
$$(-\Delta )^s u(x)\ge \frac{\delta}{|x|^{2s}},\quad x\in  B_2(0).$$

Let now~$w$ be the solution to the problem
$$
\begin{cases}
(-\Delta )^s w(x)= \dfrac{\delta}{|x|^{2s}},\quad x\in  B_1(0)\\
w(x)=0, \quad x\in \mathbb{R}^N\setminus B_1(0).
\end{cases}
$$
By comparison we obtain that $u\ge w$ in $B_1(0)$.
Therefore, in order to prove Lemma~\ref{hfksdkhfksfhkfhdjb},
it is sufficient to prove that
\begin{equation}\label{lim1432}
\lim_{|x|\rightarrow 0}w(x)=+\infty.\end{equation}
For this, we define $\tilde{w}(x):=\frac{\delta}{|x|^{2s}}* \frac{C_{N,s}}{|x|^{N-2s}}$, where  $\frac{C_{N,s}}{|x|^{N-2s}}$ is the fundamental solution of the fractional Laplace equation. So, for $|x|:=1/n$, we have that
\begin{eqnarray*}
\tilde{w}(x)&\geq &C_{N,s}\,\delta \int_{B_\tau(0)}{\frac{1}{|y|^{2s}|x-y|^{N-2s}}\,dy}\\
&\geq&\tilde{C}\int_{B_\tau(0)}{\frac{1}{|y|^{2s}\left(|y|^{N-2s}+(1/n)^{N-2s}\right)}\,dy}\rightarrow +\infty
\end{eqnarray*}
when $n\rightarrow +\infty$, that is, when $|x|\rightarrow 0$.

Also, we have that $w-\tilde{w}$ is harmonic in $B_1(0)$, and hence bounded in $B_{1/2}(0)$ (see e.g. \cite[Proposition 4.1.1]{maria} and references therein). These considerations imply~\eqref{lim1432}, as desired.
\end{proof}

We are now in the position of completing the proof
of Theorem~\ref{thm:radiality}.

\begin{proof}[Proof of Theorem~\ref{thm:radiality}]
We take~$\lambda<0$ and we introduce the following function
\begin{equation}\label{defw}
w_\lambda(x):=\left\{\begin{array}{ll}
(u-u_\lambda)^+(x), & \quad {\mbox{ if }} x\in\Sigma_\lambda, \\
(u-u_\lambda)^-(x), & \quad {\mbox{ if }}x\in\R^N\setminus\Sigma_\lambda,
\end{array} \right.
\end{equation}
where~$(u-u_\lambda)^+:=\max\{u-u_\lambda,0\}$ and~$(u-u_\lambda)^-:=\min\{u-u_\lambda,0\}$. We set
\begin{equation}\label{4.3bis}
\begin{split}
&\mathcal{S}_\lambda\,:=\, supp \,\,w_\lambda(x)\cap \Sigma_\lambda, \qquad\qquad\qquad\qquad
\mathcal{S}_\lambda^c\,:=\,\Sigma_\lambda\setminus \mathcal{S}_\lambda,\\
&\mathcal{D}_\lambda\,:=\, supp \,\,w_\lambda(x)\cap \Big(\R^N\setminus\Sigma_\lambda\Big),\qquad\qquad\,
\mathcal{D}_\lambda^c\,:=\,\Big(\R^N\setminus\Sigma_\lambda\Big)\setminus \mathcal{D}_\lambda\,.\\
\end{split}
\end{equation}
It is not difficult to see that
\begin{equation}\label{refle}
{\mbox{$\mathcal D_\lambda$ is the reflection of~$\mathcal S_\lambda$.}}
\end{equation}

Thanks to Lemma \ref{hfksdkhfksfhkfhdjb}, we have that
there exists~$\rho=\rho(\lambda)>0$ such that
\begin{equation}\nonumber
u<u_\lambda\quad \text{in }\,\, B_{\rho}(0_{\lambda})\subset\Sigma_\lambda,
\end{equation}
so that
\begin{equation}\label{nonzerosupp}
0\notin \mathcal{S}_\lambda\cup\mathcal{D}_\lambda\qquad\text{and}\qquad 0_\lambda\notin \mathcal{S}_\lambda\cup\mathcal{D}_\lambda\,.
\end{equation}

We claim  that
\begin{equation}\label{claim}
{\mbox{$w_\lambda\equiv 0$ for $\lambda<0$ with~$|\lambda|$ sufficiently large.}}
\end{equation}
To prove this, we start noticing that the function~$w_\lambda$ defined
in~\eqref{defw} belongs to~$\dot{H}^s(\mathbb{R}^N)$ and so,
recalling also \eqref{nonzerosupp}, we can use it as test
function in the weak  formulations of~\eqref{Eq:P} and~\eqref{hgdhgdhfgjssfgsf}.
We have
\begin{eqnarray}\label{eq:M1}\\\nonumber
\frac{c_{N,s}}{2}\int_{\mathbb{R}^N}\int_{\mathbb{R}^N}\frac{(u(x)-u(y))(w_\lambda(x)-
w_\lambda(y))}{|x-y|^{N+2s}}\,dx\,dy&=&\vartheta\int_{\mathbb{R}^N}\frac{u}{|x|^{2s}}w_\lambda \,dx+\int_{\mathbb{R}^N}u^{2_s^*-1}w_\lambda \,dx,
\\\nonumber
\frac{c_{N,s}}{2}\int_{\mathbb{R}^N}\int_{\mathbb{R}^N}\frac{(u_\lambda(x)-u_\lambda(y))
(w_\lambda(x)-w_\lambda(y))}{|x-y|^{N+2s}}\,dx\,dy&=&\vartheta\int_{\mathbb{R}^N}\frac{u_\lambda}
{|x_\lambda|^{2s}}w_\lambda \,dx+\int_{\mathbb{R}^N}u_\lambda^{2_s^*-1}w_\lambda \,dx.
\end{eqnarray}
Subtracting the two equations in \eqref{eq:M1} we obtain
\begin{eqnarray}\label{eq:M2}\\\nonumber
&&\frac 12 c_{N,s}\int_{\mathbb{R}^N}\int_{\mathbb{R}^N}\frac{\Big((u(x)-u_\lambda(x))-(u(y)-u_\lambda(y))\Big)
\Big(w_\lambda(x)-w_\lambda(y)\Big)}{|x-y|^{N+2s}}\,dx\,dy\\\nonumber
&=&\vartheta\int_{\mathbb{R}^N}\left(\frac{u}{|x|^{2s}}-\frac{u_\lambda}{|x_\lambda|^{2s}}\right)w_\lambda\, dx+\int_{\mathbb{R}^N}(u^{2_s^*-1}-u_\lambda^{2_s^*-1})w_\lambda\, dx\\\nonumber
&\leq&\vartheta\int_{\mathbb{R}^N}\frac{(u-u_\lambda)}{|x|^{2s}}w_\lambda \, dx+\int_{\mathbb{R}^N}(u^{2_s^*-1}-u_\lambda^{2_s^*-1})w_\lambda\,dx,
\end{eqnarray}
since $|x|\geq|x_\lambda|$ and~$w_\lambda\ge0$ in~$\Sigma_\lambda$, and~$|x|\leq|x_\lambda|$ and~$w_\lambda\le0$ outside $\Sigma_\lambda$.
On the other hand, we have
\begin{equation}\label{eq:M3}
\begin{split}
&\int_{\mathbb{R}^N}\int_{\mathbb{R}^N}
\frac{\Big((u(x)-u_\lambda(x))-(u(y)-u_\lambda(y))\Big)\Big(w_\lambda(x)-w_\lambda(y)\Big)}{|x-y|^{N+2s}}\,dx\,dy\\
=\,&\int_{\mathbb{R}^N}\int_{\mathbb{R}^N}
\frac{\Big(w_\lambda(x)-w_\lambda(y)\Big)^2}{|x-y|^{N+2s}}\,dx\,dy\\
&\quad +\int_{\mathbb{R}^N}\int_{\mathbb{R}^N}
\frac{\Big(\big(u(x)-u_\lambda(x))-(u(y)-u_\lambda(y)\big)-\big(w_\lambda(x)-w_\lambda(y)\big)\Big)\Big(w_\lambda(x)-w_\lambda(y)\Big)}{|x-y|^{N+2s}}\,dx\,dy\\
=\,& \int_{\mathbb{R}^N}\int_{\mathbb{R}^N}
\frac{\Big(w_\lambda(x)-w_\lambda(y)\Big)^2}{|x-y|^{N+2s}}\,dx\,dy+\int_{\mathbb{R}^N}\int_{\mathbb{R}^N}
\frac{\mathcal{G}(x,y)}{|x-y|^{N+2s}}\,dx\,dy,
\end{split}
\end{equation}
where
\[
\mathcal{G}(x,y)\,:=\,\Big(\big(u(x)-u_\lambda(x))-(u(y)-u_\lambda(y)\big)-
\big(w_\lambda(x)-w_\lambda(y)\big)\Big)\Big(w_\lambda(x)-w_\lambda(y)\Big)\,.
\]

Now, we prove that
\begin{equation}\label{dfjhgjgfjfdecom}
\int_{\mathbb{R}^N}\int_{\mathbb{R}^N}
\frac{\mathcal{G}(x,y)}{|x-y|^{N+2s}}\,dx\,dy\geq 0\,.
\end{equation}
To check this, we use the decomposition
\begin{equation}\nonumber
\begin{split}
\mathbb{R}^N\times \mathbb{R}^N\,=\,\left( \mathcal{S}_\lambda\cup \mathcal{S}_\lambda^c\cup\mathcal{D}_\lambda\cup
\mathcal{D}_\lambda^c\right)\times \left( \mathcal{S}_\lambda\cup \mathcal{S}_\lambda^c\cup\mathcal{D}_\lambda\cup
\mathcal{D}_\lambda^c\right),
\end{split}
\end{equation}
where~$\mathcal{S}_\lambda$, $\mathcal{S}_\lambda^c$, $\mathcal{D}_\lambda$
and~$\mathcal{D}_\lambda^c$ have been introduced in~\eqref{4.3bis}.
By construction, it follows that
\begin{center}
\begin{equation}\nonumber
\begin{split}
&\mathcal{G}(x,y)= \Big[-\big(u(x)-u_\lambda(x)\big)w_\lambda(y)\Big]\quad\text{in}\quad \big(\mathcal{S}_\lambda^c \times \mathcal{S}_\lambda\big),\\
&\mathcal{G}(x,y)= \Big[-\big(u(x)-u_\lambda(x)\big)w_\lambda(y)\Big]\quad\text{in}\quad \big(\mathcal{S}_\lambda^c \times \mathcal{D}_\lambda\big),\\
&\mathcal{G}(x,y)= \Big[-\big(u(y)-u_\lambda(y)\big)w_\lambda(x)\Big]\quad\text{in}\quad \big(\mathcal{S}_\lambda \times \mathcal{S}_\lambda^c\big),\\
&\mathcal{G}(x,y)= \Big[-\big(u(y)-u_\lambda(y)\big)w_\lambda(x)\Big]\quad\text{in}\quad \big(\mathcal{S}_\lambda \times \mathcal{D}_\lambda^c\big),\\
&\mathcal{G}(x,y)= \Big[-\big(u(x)-u_\lambda(x)\big)w_\lambda(y)\Big]\quad\text{in}\quad \big(\mathcal{D}_\lambda^c \times \mathcal{S}_\lambda\big),\\
&\mathcal{G}(x,y)= \Big[-\big(u(x)-u_\lambda(x)\big)w_\lambda(y)\Big]\quad\text{in}\quad \big(\mathcal{D}_\lambda^c \times \mathcal{D}_\lambda\big),\\
&\mathcal{G}(x,y)= \Big[-\big(u(y)-u_\lambda(y)\big)w_\lambda(x)\Big]\quad\text{in}\quad \big(\mathcal{D}_\lambda \times \mathcal{S}_\lambda^c\big),\\
&\mathcal{G}(x,y)= \Big[-\big(u(y)-u_\lambda(y)\big)w_\lambda(x)\Big]\quad\text{in}\quad \big(\mathcal{D}_\lambda \times \mathcal{D}_\lambda^c\big)\\
{\mbox{and }}\quad  &\mathcal{G}(x,y)= 0\qquad \qquad\qquad\qquad\qquad\quad\,\,\,\,\text{elsewhere}\,.
\end{split}
\end{equation}
\end{center}

We have that
\begin{equation}\label{dasdwefg}
\int_{\mathcal{S}_\lambda^c}\int_{\mathcal{S}_\lambda}
\frac{\mathcal{G}(x,y)}{|x-y|^{N+2s}}\,dx\,dy+
\int_{\mathcal{S}_\lambda^c}\int_{\mathcal{D}_\lambda}
\frac{\mathcal{G}(x,y)}{|x-y|^{N+2s}}\,dx\,dy
\geq 0\,.
\end{equation}
Indeed, notice that, if~$x\in\mathcal S^c_\lambda$ and~$y\in\mathcal S_\lambda$,
then~$\mathcal G(x,y)\ge0$, and moreover
\begin{eqnarray*}
\mathcal G(x,y)&=&-\big(u(x)-u_\lambda(x)\big)w_\lambda(y)\\
&=&-\big(u(x)-u_\lambda(x)\big)\big(u(y)-u_\lambda(y)\big)\\
&=&-\big(u(x)-u_\lambda(x)\big)\big(u_\lambda(y_\lambda)-u(y_\lambda)\big)\\
&=& \big(u(x)-u_\lambda(x)\big)w_\lambda(y_\lambda)\\
&=&-\mathcal G(x,y_\lambda).
\end{eqnarray*}
Also, we have that~$|x-y|\leq|x-y_\lambda|$ in~$\mathcal{S}_\lambda^c \times \mathcal{S}_\lambda$. Therefore, using also~\eqref{refle}, we have
\begin{eqnarray*}
&&\int_{\mathcal{S}_\lambda^c}\int_{\mathcal{S}_\lambda}
\frac{\mathcal{G}(x,y)}{|x-y|^{N+2s}}\,dx\,dy+
\int_{\mathcal{S}_\lambda^c}\int_{\mathcal{D}_\lambda}
\frac{\mathcal{G}(x,y)}{|x-y|^{N+2s}}\,dx\,dy\\
&&\qquad =\int_{\mathcal{S}_\lambda^c}\int_{\mathcal{S}_\lambda}
\frac{\mathcal{G}(x,y)}{|x-y|^{N+2s}}\,dx\,dy+
\int_{\mathcal{S}_\lambda^c}\int_{\mathcal{S}_\lambda}
\frac{\mathcal{G}(x,y_\lambda)}{|x-y_\lambda|^{N+2s}}\,dx\,dy\\
&&\qquad =\int_{\mathcal{S}_\lambda^c}\int_{\mathcal{S}_\lambda}
\mathcal G(x,y)\left[\frac{1}{|x-y|^{N+2s}}-\frac{1}{|x-y_\lambda|^{N+2s}}\right]\,dx\,dy\ge 0.
\end{eqnarray*}
which shows~\eqref{dasdwefg}.

Similarly, one can prove that
\[
\int_{\mathcal{S}_\lambda}\int_{\mathcal{S}_\lambda^c}
\frac{\mathcal{G}(x,y)}{|x-y|^{N+2s}}\,dx\,dy+
\int_{\mathcal{S}_\lambda}\int_{\mathcal{D}_\lambda^c}
\frac{\mathcal{G}(x,y)}{|x-y|^{N+2s}}\,dx\,dy
\geq 0,
\]
\[
\int_{\mathcal{D}_\lambda^c}\int_{\mathcal{S}_\lambda}
\frac{\mathcal{G}(x,y)}{|x-y|^{N+2s}}\,dx\,dy+
\int_{\mathcal{D}_\lambda^c}\int_{\mathcal{D}_\lambda}
\frac{\mathcal{G}(x,y)}{|x-y|^{N+2s}}\,dx\,dy
\geq 0
\]
and
\[
\int_{\mathcal{D}_\lambda}\int_{\mathcal{S}_\lambda^c}
\frac{\mathcal{G}(x,y)}{|x-y|^{N+2s}}\,dx\,dy+
\int_{\mathcal{D}_\lambda}\int_{\mathcal{D}_\lambda^c}
\frac{\mathcal{G}(x,y)}{|x-y|^{N+2s}}\,dx\,dy
\geq 0.
\]
Collecting the estimates above we obtain \eqref{dfjhgjgfjfdecom}.

Hence, from~\eqref{eq:M3} and~\eqref{dfjhgjgfjfdecom}, we deduce that
\begin{equation}\label{eq:M333}
\begin{split}
&\frac 12 c_{N,s}\int_{\mathbb{R}^N}\int_{\mathbb{R}^N}
\frac{\Big((u(x)-u_\lambda(x))-(u(y)-u_\lambda(y))\Big)\Big(w_\lambda(x)-w_\lambda(y)\Big)}{|x-y|^{N+2s}}\,dx\,dy\\
&\geq\frac 12 c_{N,s}\int_{\mathbb{R}^N}\int_{\mathbb{R}^N}\frac{(w_\lambda(x)-w_\lambda(y))^2}{|x-y|^{N+2s}}\,dx\,dy\,.
\end{split}
\end{equation}
Using this,~\eqref{eq:M2} and the fact that~$(u-u_\lambda)w_\lambda=w_\lambda^2$
in~$\mathbb{R}^N$, we obtain
\begin{equation}\label{eq:M22}
\frac{c_{N,s}}{2}\int_{\mathbb{R}^N}\int_{\mathbb{R}^N}\frac{(w_\lambda(x)-w_\lambda(y))^2}{|x-y|^{N+2s}}\,dx\,dy\le
\vartheta\int_{\mathbb{R}^N}\frac{w_\lambda^2}{|x|^{2s}} \,dx+\int_{\mathbb{R}^N}(u^{2_s^*-1}-u_\lambda^{2_s^*-1})w_\lambda \,dx.
\end{equation}
For the first term in the right hand side of \eqref{eq:M22}, we use Hardy inequality and we get
\begin{equation}\label{eq:M4}
\vartheta\int_{\mathbb{R}^N}\frac{w_\lambda^2}{|x|^{2s}} \,dx\leq \frac{\vartheta}{\Lambda_{N,s}}\frac{c_{N,s}}{2}\int_{\mathbb{R}^N}\int_{\mathbb{R}^N}\frac{(w_\lambda(x)-w_\lambda(y))^2}{|x-y|^{N+2s}}\,dx\,dy.
\end{equation}
For the second term in the right hand side of \eqref{eq:M22}, we first notice
that, thanks to~\eqref{refle},
$$ \int_{\mathcal{S}_\lambda}u^{2_s^*}\,dx=\int_{\mathcal{D}_\lambda} u_\lambda^{2_s^*}\,dx. $$
Moreover, $w_\lambda=0$ in~$\mathcal S^c_\lambda\cup\mathcal D^c_\lambda$.
Therefore, using also Lagrange Theorem and H\"older inequality, we have
\begin{equation}\label{eq:M5}
\begin{split}
&\int_{\mathbb{R}^N}(u^{2_s^*-1}-u_\lambda^{2_s^*-1})w_\lambda\,dx\\
=\,&\int_{\mathcal{S}_\lambda}(u^{2_s^*-1}-u_\lambda^{2_s^*-1})w_\lambda\,dx
+\int_{\mathcal{D}_\lambda}(u^{2_s^*-1}-u_\lambda^{2_s^*-1})w_\lambda\,dx\\
\leq\,& C_1\int_{\mathcal{S}_\lambda}u^{2^*_s-2}\cdot w_\lambda^2\,dx
+C_1\int_{\mathcal{D}_\lambda}u_\lambda^{2^*_s-2}\cdot w_\lambda^2\,dx\\
\leq \,& C_1\left(\int_{\mathcal{S}_\lambda}u^{2_s^*}\,dx\right)^{\frac{2^*_s-2}{2_s^*}}
\left(\int_{\mathcal{S}_\lambda}w_\lambda^{2^*_s}\,dx\right)^{\frac{2}{2^*_s}} +
C_1\left(\int_{\mathcal{D}_\lambda}u_\lambda^{2_s^*}\,dx\right)^{\frac{2^*_s-2}{2_s^*}}
\left(\int_{\mathcal{D}_\lambda}w_\lambda^{2^*_s}\,dx\right)^{\frac{2}{2^*_s}} \\
\leq\,& C_2\left(\int_{\mathcal{S}_\lambda}u^{2_s^*}\,dx\right)^{\frac{2^*_s-2}{2_s^*}}
\left(\int_{\mathbb{R}^N} w_\lambda^{2^*_s}\,dx\right)^{\frac{2}{2^*_s}}\\
\leq\,& C_3 \left(\int_{\mathcal{S}_\lambda}u^{2_s^*}\,dx\right)^{\frac{2^*_s-2}{2_s^*}}
\int_{\mathbb{R}^N}\int_{\mathbb{R}^N} \frac{(w_\lambda(x)-w_\lambda(y))^2}{|x-y|^{N+2s}}\,dx\,dy,
\end{split}
\end{equation}
where we have also used the Sobolev embedding (see, for instance, Theorem~$6.5$ in~\cite{DPV12}). Notice that the constants~$C_1$,~$C_2$
and~$C_3$ are positive and independent of~$\lambda$.

Collecting the inequalities in~\eqref{eq:M22},~\eqref{eq:M4} and~\eqref{eq:M5},  we obtain
\begin{eqnarray}\label{eq:M6}\\\nonumber
&&\left(\frac{c_{N,s}}{2}-\frac{\vartheta}{\Lambda_{N,s}}\frac{c_{N,s}}{2}\right)\int_{\mathbb{R}^N}\int_{\mathbb{R}^N}\frac{(w_\lambda(x)-w_\lambda(y))^2}{|x-y|^{N+2s}}\,dx\,dy
\\\nonumber&\leq& C_3\left(\int_{\mathcal{S}_\lambda}u^{2_s^*}\,dx\right)^{\frac{2^*_s-2}{2_s^*}}\int_{\mathbb{R}^N}\int_{\mathbb{R}^N}\frac{(w_\lambda(x)-w_\lambda(y))^2}{|x-y|^{N+2s}}\,dx\,dy.
\end{eqnarray}

Since  $u\in\dot{H}^s(\mathbb{R}^N)$ (and therefore in~$L^{2_s^*}(\mathbb{R}^N)$), there exists~$R>0$ such that
for~$\lambda<-R$ we have
$$ C_3 \left(\int_{\mathcal{S}_\lambda}u^{2_s^*}\,dx\right)^{\frac{2^*_s-2}{2_s^*}}\leq
C_3 \left(\int_{\Sigma_\lambda}u^{2_s^*}\,dx\right)^{\frac{2^*_s-2}{2_s^*}}\le  \frac12\left(\frac{c_{N,s}}{2}-\frac{\vartheta}{\Lambda_{N,s}}\frac{c_{N,s}}{2}\right).$$
This and~\eqref{eq:M6} give that
$$\int_{\mathbb{R}^N}\int_{\mathbb{R}^N}\frac{(w_\lambda(x)-w_\lambda(y))^2}{|x-y|^{N+2s}}\,dx\,dy= 0\,.$$
This implies that  $w_\lambda$ is constant and the claim~\eqref{claim}
follows since~$w_\lambda$ is zero on $\{x_1=\lambda\}$.

Now we define the set
$$
\Lambda:=\{\lambda\in \mathbb{R}\,:\, u\leq u_\mu \,\text{in }\, \Sigma_\mu\,\, \forall \mu \leq \lambda \}.$$
Notice that~\eqref{claim} implies that~$\Lambda\neq\emptyset$,
and therefore we can consider
\begin{equation}\label{eq:sup}
\bar \lambda:=\sup\Lambda.
\end{equation}
We will show that
\begin{equation}\label{lambda zero}
\bar\lambda=0.
\end{equation}
Let us  assume by contradiction that $\bar\lambda <0$.
Now, in this case,  we are going to show that we can move the plane a little further to the right reaching a contradiction with the definition \eqref{eq:sup}.

First, we prove that
\begin{equation}\label{claim2}
{\mbox{$u<u_{\bar\lambda}$ in~$\Sigma_{\bar\lambda}$.}}
\end{equation}
Indeed, by continuity, we have that $u\leq u_{\bar\lambda}$ in~$\Sigma_{\bar\lambda}$ (say outside the reflected point of the origin $0_{\bar\lambda}$).
On the other hand, Lemma~\ref{hfksdkhfksfhkfhdjb} implies that
there exists $\rho>0$ such that
\begin{equation}\label{bcnbvnvbcjsusehiks}
u<u_{\bar\lambda}\quad \text{in }\,\, B_{\rho}(0_{\bar\lambda})\subset\Sigma_{\bar\lambda}.
\end{equation}

We take~$x_0\in \Sigma_{\bar\lambda}\setminus \{0_{\bar\lambda}\}$ and we fix $\bar\rho>0$ such that~$B_{\bar\rho}(x_0)\subset \Sigma_{\bar\lambda}\setminus \{0_{\bar\lambda}\}$.

We set
\[
\omega_{\bar\lambda}\,:=\,u_{\bar\lambda}-u.
\]
Notice that~$\omega_{\bar\lambda}\in\dot{H}^s(\mathbb{R}^N)\cap C(B_{\bar\rho}(x_0))$ and~$\omega_{\bar\lambda}\ge0$ in~$\Sigma_{\bar\lambda}$.
Moreover, since~$|x|\ge|x_{\bar\lambda}|$ and~$u\le u_{\bar\lambda}$ in~$\Sigma_{\bar\lambda}$, using the
weak formulations of~\eqref{Eq:P} and~\eqref{hgdhgdhfgjssfgsf},
we have that
\begin{eqnarray*}
&&\frac{c_{N,s}}{2}\int_{\mathbb{R}^N}\int_{\mathbb{R}^N}\frac{(\omega_{\bar\lambda}(x)-\omega_{\bar\lambda}(y))(\varphi(x)-\varphi(y))}{|x-y|^{N+2s}}\,dx\,dy\\
&=& \frac{c_{N,s}}{2} \int_{\mathbb{R}^N}\int_{\mathbb{R}^N}\frac{(u_{\bar\lambda}(x)-u_{\bar\lambda}(y))(\varphi(x)-\varphi(y))}{|x-y|^{N+2s}}\,dx\,dy
\\&&\qquad -
\frac{c_{N,s}}{2} \int_{\mathbb{R}^N}\int_{\mathbb{R}^N}\frac{(u(x)-u(y))(\varphi(x)-\varphi(y))}{|x-y|^{N+2s}}\,dx\,dy \\
&=& \vartheta\int_{\mathbb{R}^N}\frac{u_{\bar\lambda}(x)}{|x_{\bar\lambda}|^{2s}}\,\varphi(x)\,dx +\int_{\mathbb{R}^N}u_{\bar\lambda}^{2^*_s-1}(x)\varphi(x)\,dx \\&&\qquad
-
\vartheta\int_{\mathbb{R}^N}\frac{u(x)}{|x|^{2s}}\,\varphi(x)\,dx - \int_{\mathbb{R}^N}u^{2^*_s-1}(x)\varphi(x)\,dx \\
&\ge &
\vartheta\int_{\mathbb{R}^N}\frac{u_{\bar\lambda}(x)-u(x)}{|x|^{2s}}\,\varphi(x)\,dx +\int_{\mathbb{R}^N}\left(
u_{\bar\lambda}^{2^*_s-1}(x)- u^{2^*_s-1}(x)\right)\varphi(x)\,dx\\
&\ge & 0,
\end{eqnarray*}
for any nonnegative test function~$\varphi$ with compact support in~$B_\rho(x_0)$. This implies that
$$ (-\Delta)^s \omega_{\bar\lambda}\ge 0 \quad {\mbox{ in }} B_\rho(x_0)$$
in the weak sense.
Therefore,~$\omega_{\bar\lambda}$ satisfies the hypotheses of Proposition~\ref{gfjhbgfjhfjhjnbncbncbcnvakodkod}, and so
either~$\omega_{\bar\lambda}\equiv 0$ in~$\mathbb{R}^N$ or~$\omega_{\bar\lambda}>0$ in~$B_{\bar\rho}(x_0)$.

If~$\omega_{\bar\lambda}\equiv 0$ in~$\mathbb{R}^N$,
then~$u=u_{\bar\lambda}$ in~$\mathbb{R}^N$, which contradicts~\eqref{bcnbvnvbcjsusehiks}.
Therefore~$\omega_{\bar\lambda}>0$ in~$B_{\bar\rho}(x_0)$,
which implies that~$u<u_{\bar\lambda}$ in~$B_{\bar\rho}(x_0)$.
Since~$x_0$ is an arbitrary point in~$\Sigma_{\bar\lambda}\setminus\{0_{\bar\lambda}\}$, this implies~\eqref{claim2}.

Now, notice that the inequality in~\eqref{eq:M6} holds for any~$\lambda<0$
and the constant~$C_3$ there is independent of~$\lambda$.
Moreover, since~$\bar\lambda<0$, there exists~$\varepsilon_1>0$
such that~$\bar\lambda+\varepsilon<0$ for any~$\varepsilon\in(0,\varepsilon_1)$.
Recalling the notation introduced in~\eqref{defw} and~\eqref{4.3bis},
we consider the function~$w_{\bar\lambda+\epsilon}$.
Using the same notation as above let us consider $w_{\bar\lambda+\varepsilon}$ so that
\[
supp\,\,w_{\bar\lambda+\varepsilon}\,\equiv\,\mathcal{S}_{\bar\lambda+\varepsilon}
\cup\mathcal{D}_{\bar\lambda+\varepsilon}\,.
\]
Exploiting the fact that $u<u_{\bar\lambda}$ in $\Sigma_{\bar\lambda}$ and the fact that the solution $u$ is continuous in $\mathbb{R}^N\setminus\{0\}$ and \eqref{eq:liminfinito}, we deduce that:\\
\noindent given any $R>0$ (large) and $\delta>0$ (small) we can fix $\bar\varepsilon=\bar\varepsilon(R,\delta)>0$ such that
\begin{equation}
\mathcal{S}_{\bar\lambda+\varepsilon}\cap\Sigma_{\bar\lambda-\delta}
\subset \mathbb{R}^N\setminus B_R(0)\qquad\text{for any}\quad0\leq \varepsilon\leq\bar\varepsilon\,.
\end{equation}
We repeat now the argument above using  $w_{\bar\lambda+\varepsilon}$ as test function in the same fashion as we did using $w_\lambda$ and get again
\begin{eqnarray}\label{eq:M7}\\\nonumber
&&\left(\frac{c_{N,s}}{2}-\frac{\vartheta}{\Lambda_{N,s}}\frac{c_{N,s}}{2}\right)
\int_{\mathbb{R}^N}\int_{\mathbb{R}^N}\frac{(w_{\bar\lambda+\varepsilon}(x)-w_{\bar\lambda+\varepsilon}(y))^2}{|x-y|^{N+2s}}\,dx\,dy
\\\nonumber&\leq&\bar C(N,s)
\left(\int_{\mathcal{S}_{\bar\lambda+\varepsilon}}u^{2_s^*}\,dx\right)^{\frac{2^*_s-2}{2_s^*}}
\int_{\mathbb{R}^N}\int_{\mathbb{R}^N}\frac{(w_{\bar\lambda+\varepsilon}(x)-w_{\bar\lambda+\varepsilon}(y))^2}{|x-y|^{N+2s}}\,dx\,dy.
\end{eqnarray}
Since
\begin{equation}
\begin{split}
\int_{\mathcal{S}_{\bar\lambda+\varepsilon}}u^{2_s^*}\,dx &\leq\int_{\mathcal{S}_{\bar\lambda+\varepsilon}\cap\Sigma_{\bar\lambda-\delta}}u^{2_s^*}\,dx
+\int_{\Sigma_{\bar\lambda+\varepsilon}\setminus \Sigma_{\bar\lambda-\delta}}u^{2_s^*}\,dx\\
&=\int_{\mathbb{R}^N\setminus B_R(0)}u^{2_s^*}\,dx
+\int_{\Sigma_{\bar\lambda+\varepsilon}\setminus \Sigma_{\bar\lambda-\delta}}u^{2_s^*}\,dx\\
\end{split}
\end{equation}
for $R$ large and $\delta$ small, choosing $\bar\varepsilon(R,\delta)$ as above and eventually reducing it, we can assume that
$$\bar C(N,s)\left(\int_{\mathcal{S}_{\bar\lambda+\varepsilon}}u^{2_s^*}\,dx\right)^{\frac{2^*_s-2}{2_s^*}}<
\left(\frac{c_{N,s}}{2}-\frac{\vartheta}{\Lambda_{N,s}}\frac{c_{N,s}}{2}\right),$$
for $\bar C(N,s)$ given by \eqref{eq:M7}.
Then from \eqref{eq:M7} we reach that $w_{\bar\lambda+\varepsilon}=0$ and a this contradiction to \eqref{eq:sup}. Therefore
\[
\bar\lambda=0\,.
\]

Finally, the symmetry (and monotonicity) in the $x_1$-direction follows as standard repeating the argument in the $(-x_1)$-direction. The radial symmetry result (and the monotonicity of the solution) follows as well performing the \textit{Moving Plane Method} in any direction $\nu\in S^{N-1}$.
\end{proof}

\section{Asymptotic analysis of solutions to equation \eqref{Eq:P}}\label{sec:Asymptotic}

In this section we investigate the behaviour of a solution of \eqref{Eq:P}
near the origin and at infinity and we prove Theorem~\ref{thm:main}.

\subsection{A representation formula and an equivalent nonlocal problem}
In order to study the behaviour of the solutions to \eqref{Eq:P} near the origin
and at infinity,
we are going to use a representation result by Frank, Lieb and Seiringer,
in particular equality $(4.3)$ proved in \cite[pag. 935]{FLS}.  We have the following:

\begin{lemma}\label{lemma:fls} (Ground State Representation) Let $0<\alpha<{(N-2s)}/{2}$ and let
$u\in C^{\infty}_0(\mathbb{R}^N \setminus \{0\})$. Set also $v_{\alpha}(x):=|x|^{\alpha}u(x)$.
Then
\begin{eqnarray}\label{eq:fralieser}
&&\frac 12 c_{N,s}\int_{\mathbb{R}^N}\int_{\mathbb{R}^N}\frac{|u(x)-u(y)|^2}{|x-y|^{N+2s}}\,dx\,dy
-(\Lambda_{N,s}+\Phi_{s,N}(\alpha))\int_{\mathbb{R}^N}|x|^{-2s}|u(x)|^2\,dx\\\nonumber
&=&\frac{c_{N,s}}{2}\int_{\mathbb{R}^N}\int_{\mathbb{R}^N}\frac{|v_{\alpha}(x)-v_{\alpha}(y)|^2}{|x-y|^{N+2s}}\frac{dx}{|x|^\alpha}\frac{dy}{|y|^\alpha},
\end{eqnarray}
where $ \Phi_{s,N}(\cdot)$ is given by
\begin{equation}\label{eq:gamma}
\Phi_{s,N}(\alpha)=2^{2s}\left(\frac{\Gamma(\frac{\alpha+2s}{2})\Gamma(\frac{N-\alpha}{2})}{\Gamma(\frac{N-\alpha-2s}{2})\Gamma(\frac{\alpha}{2})}-\frac{\Gamma^2(\frac{N+2s}{4})}{\Gamma^2(\frac{N-2s}{4})}\right).
\end{equation}
\end{lemma}

\begin{remark}
The result in Lemma \ref{lemma:fls} in particular shows that he constant $\Lambda_{N,s}$ in the Hardy-Sobolev inequality~\eqref{Hardy} is optimal and is not attained. This is the peculiar spectral behavior of the Hardy potential, motivated by the fact that the potential $|x|^{-2s}$ is in the Marcinkievicz space
$\mathcal{M}^{\frac{N}{2s},\infty}$ but  not in the space $L^{\frac{N}{2s}}_{loc}(\mathbb{R}^N)$.  See for details  Remark 4.2  in \cite{FLS}.
\end{remark}

Also, see \cite[Lemma 3.2]{FLS}, the function $\Phi_{s,N}(\cdot)$ is negative and strictly
increasing in $(0, (N-2s)/2)$ with  $\Phi_{s,N}((N-2s)/2)=0$,  that is

\begin{proposition}\cite[Lemma 3.2]{FLS}\label{pr:surjective}
Consider the function
\begin{eqnarray}\nonumber
\Psi_{s,N}:\Big[0, \frac{N-2s}{2}\Big]\,\, &\rightarrow& \,\, [0, \Lambda_{N,s}]\\\nonumber
\alpha \,\, &\rightarrow& \,\, \Psi_{s,N}(\alpha):=\Lambda_{N,s}+\Phi_{s,N}(\alpha),
\end{eqnarray}
with $\Phi_{s,N}(\cdot)$ defined in \eqref{eq:gamma}. Then $\Psi_{s,N}$ is strictly increasing and surjective.
\end{proposition}

Given $\theta \in (0,\Lambda_{N,s})$ in  \eqref{Eq:P},  by Proposition \ref{pr:surjective},   it follows the existence of  a unique $\alpha \in (0, (N-2s)/2)$ such that
\begin{equation}\label{eq:alpha}
\Psi_{s,N}(\alpha)=\theta.
\end{equation}
We will denote by $\alpha_\theta$ the solution of \eqref{eq:alpha}.

Then, by \eqref{1.1bis} with $\varphi:=u$ and \eqref{eq:fralieser}, we get
\begin{eqnarray}\label{eq:cucucucrucucu}\nonumber
\int_{\mathbb{R}^N}u^{2^*_s}\,dx&=&\frac 12 c_{N,s}\int_{\mathbb{R}^N}\int_{\mathbb{R}^N}
\frac{|u(x)-u(y)|^2}{|x-y|^{N+2s}}\,dx\,dy-\theta\int_{\mathbb{R}^N}|x|^{-2s}|u(x)|^2\,dx
\\&=&\frac{C_{N,s}}{2}\int_{\mathbb{R}^N}\int_{\mathbb{R}^N}\frac{|v_{\alpha}(x)-v_{\alpha}(y)|^2}
{|x-y|^{N+2s}}\frac{dx}{|x|^\alpha}\frac{dy}{|y|^\alpha}.
\end{eqnarray}
On the other hand,
recalling that
\begin{equation}\label{valfa}
v_{\alpha}(x):=|x|^{\alpha}u(x),
\end{equation} with $\alpha=\alpha_\theta$,
from \eqref{eq:cucucucrucucu} we get
\begin{equation}\label{eq:equivalent}
\frac{C_{N,s}}{2}\int_{\mathbb{R}^N}\int_{\mathbb{R}^N}\frac{|v_{\alpha}(x)-v_{\alpha}(y)|^2}{|x-y|^{N+2s}}\frac{dx}{|x|^\alpha}\frac{dy}{|y|^\alpha}=\int_{\mathbb{R}^N}\frac{v^{2^*_s}_\alpha(x)}{|x|^{\alpha \cdot 2^*_s}}\,dx.
\end{equation}
This suggests to define the space $\dot H^{s,\alpha}(\mathbb{R}^N)$
as the closure of $C^{\infty}_0(\mathbb R^N)$ with respect to the norm
$$\| \phi\|_{\dot H^{s,\alpha}(\mathbb{R}^N)}= \left(\int_{\R^N}\frac{|\phi(x)|^{2^*_s}}{|x|^{\alpha 2^*_s}}
\,dx\right)^{\frac{1}{2^*_s}} +
\left(\int_{\mathbb{R}^N}\int_{\mathbb{R}^N}\frac{|\phi(x)-\phi(y)|^2}{|x-y|^{N+2s}}\frac{dx}
{|x|^\alpha}\frac{dy}{|y|^\alpha}\right)^{\frac12}.$$
We also define
$$ \dot W^{s,\alpha}(\R^N):=\{\phi:\R^N\to\R {\mbox{ measurable }}: \| \phi\|_{\dot H^{s,\alpha}(\mathbb{R}^N)}<+\infty\}.$$
Note that, following e.g.~\cite{DV}, one has that
the space $\dot W^{s,\alpha}(\R^N)$ coincides with $\dot H^{s,\alpha}(\R^N)$.

\begin{remark} \label{rem:density}
Notice that Lemma \ref{lemma:fls} says that equality \eqref{eq:fralieser} holds true
for functions  $u\in C^{\infty}_0(\mathbb{R}^N \setminus \{0\})$.
Actually, by a density argument one can prove that it holds for any $u\in \dot{H}^s(\R^N)$.

Indeed, one first approximates a function~$u\in \dot{H}^s(\R^N)$
with a function~$u_\epsilon\in C^{\infty}_0(\mathbb{R}^N)$.
Then, one uses a standard cut-off argument near zero.
In this way, one can see that in the left hand side of~\eqref{eq:fralieser}
it is possible to pass to the limit.

In order to pass to the limit also in the right hand side of~\eqref{eq:fralieser}, one needs to notice that, by the representation formula, Cauchy sequences are sent into Cauchy sequences and
we are working in Hilbert spaces. Therefore, the conclusion follows
by observing that on the left hand side we have a Cauchy sequence since it
is convergent.
\end{remark}

As a consequence of the \textit{ground state representation} given by Lemma~\ref{lemma:fls}, we will transform our problem~\eqref{Eq:P}
into another nonlocal problem in weighted spaces. Namely, we consider~$u\in \dot{H}^s(\mathbb{R}^N)$ that is a solution of the problem
$$
(-\Delta)^s u=\vartheta\frac{u}{|x|^{2s}}+u^{2_s^*-1} \quad\text{in}\,\, \mathbb{R}^N\setminus{\{0\}},
$$
and we set $v_{\alpha}(x):=|x|^{\alpha}u(x)$ with $\alpha=\alpha_\theta$ given by \eqref{eq:alpha}.

By Lemma \ref{lemma:fls}, Remark~\ref{rem:density} and~\cite{AB},
it follows that~$v_{\alpha}\in  \dot H^{s,\alpha}(\R^N)$.

Furthermore, we introduce the operator~$(-\Delta_{\alpha})^s$, defined as duality product
\begin{equation}\label{operatoralpha}
<(-\Delta_{\alpha})^s v,\phi>=\int_{\mathbb{R}^N}\int_{\mathbb{R}^N}
\frac{(v(x)-v(y))(\phi(x)-\phi(y))}{|x-y|^{N+2s}}\frac{dx}{|x|^\alpha}\frac{dy}{|y|^\alpha},
\end{equation}
for any~$\phi\in \dot H^{s,\alpha}(\R^N)$.
With this notation, we have that~$v_\alpha$ is a weak solution to
\begin{equation}\label{eq:valfa}
(-\Delta_{\alpha})^s v=\frac{v^{2_s^*-1}}{|x|^{\alpha\cdot 2^*_s}} \quad\text{in}\,\, \mathbb{R}^N,
\end{equation}
namely for any $\phi\in \dot H^{s,\alpha}(\R^N)$,
$$\int_{\mathbb{R}^N}\int_{\mathbb{R}^N}
\frac{(v(x)-v(y))(\phi(x)-\phi(y))}{|x-y|^{N+2s}}\frac{dx}{|x|^\alpha}\frac{dy}{|y|^\alpha}=\int_{\R^N}\frac{v^{2_s^*-1}(x)}{|x|^{\alpha\cdot 2^*_s}}
\phi(x)\,dx.$$
Summarizing, with the \textit{ground state representation} we have hidden the Hardy potential and the cost is that we now have to handle~$(-\Delta_{\alpha})^s$, that is an elliptic operator with singular coefficients.

On the other hand, to get the exact behavior of the solution $u$ to~\eqref{Eq:P} near the origin and at infinity, it is sufficient
to get an $L^\infty$ estimate and the Harnack inequality for $v_\alpha$.
This is the main goal of the forthcoming Subsections~\ref{sec:infi}
and~\ref{sec:harn}.

\medskip

\subsection{A regularity result: the $L^{\infty}$ estimate}\label{sec:infi}
In this section we prove a regularity result for weak solution of \eqref{eq:valfa}.
More precisely:

\begin{proposition}\label{pro: linfty}
Let~$\alpha\in(0,(N-2s)/2)$.
Let $v\in \dot H^{s,\alpha}(\R^N)$ be a nonnegative weak solution of
\begin{equation*}
(-\Delta_{\alpha})^s v=\frac{v^{2_s^*-1}}{|x|^{\alpha\cdot 2^*_s}} \quad\text{in}\,\, \mathbb{R}^N.
\end{equation*}
Then $v\in L^{\infty}(\mathbb R^N)$.
\end{proposition}
\begin{proof}
Let us define for $\beta\geq 1$ and $T>0$
$$\phi(t)=\phi_{\beta,T}(t)=\begin{cases}t^\beta, & \text{if }\,\, 0\leq t\leq T \\
\beta T^{\beta-1}(t-T)+T^\beta, & \text{if }\,\, t>T.
\end{cases}$$
We  observe that (as in the case of the standard fractional laplacian $(-\Delta)^s(\cdot)$, see \cite{LPPS}) it holds in the weak distributional meaning  that
\begin{equation}\label{eq:convex}(-\Delta_{\alpha})^s\phi(v)\leq \phi'(v) (-\Delta_{\alpha})^sv, \,\,v\in \dot H^{s,\alpha}(\mathbb{R}^N).\end{equation}

Since $\phi_{\beta,T}$ is a Lipschitz function it follows that $\phi_{\beta,T}(v)\in \dot H^{s,\alpha}(\mathbb{R}^N)$.
By using the weighted Sobolev inequality we have
\begin{eqnarray}\label{eq:left}
&&\frac{c_{N,s}}{2}
\int_{\mathbb R^N} \int_{\mathbb R^N}\frac{|\phi(v(x))-\phi(v(y))|^2}{|x-y|^{N+2s}}\frac{dx}{|x|^{\alpha}}\frac{dy}{|y|^{\alpha}}\\\nonumber
&\geq&\frac{c_{N,s}}{2}S(N,s,\alpha)\left(\int_{\mathbb{R}^N}|\phi(v)|^{2^*_s}\frac{dx}{|x|^{\alpha\cdot 2^*_s}}\right)^{\frac{2}{2^*_s}}.
\end{eqnarray}
On the other hand, using \eqref{eq:convex},  we  get
\begin{eqnarray}\label{eq: right}
&&\int_{\mathbb{R}^N}\phi(v)(-\Delta_\alpha)^s\phi(v)\leq\int_{\mathbb{R}^N}\phi(v)\phi'(v)(-\Delta_\alpha)^sv\\\nonumber&=&\int_{\mathbb{R}^N}\phi(v)\phi'(v)v^{2_s^*-1}\frac{dx}{|x|^{\alpha\cdot 2^*_s}}\leq\beta\int_{\mathbb{R}^N} (\phi (v))^2v^{2_s^*-2}\frac{dx}{|x|^{\alpha\cdot 2^*_s}},
\end{eqnarray}
where we used that $t\phi'(t)\leq \beta \phi (t)$. From  \eqref{eq:left} and \eqref{eq: right} we obtain
\begin{equation}\label{eq:leftright}
\left(\int_{\mathbb{R}^N}|\phi(v)|^{2^*_s}\frac{dx}{|x|^{\alpha\cdot 2^*_s}}\right)^{\frac{2}{2^*_s}}\leq C\beta \int_{\mathbb{R}^N} (\phi (v))^2v^{2_s^*-2}\frac{dx}{|x|^{\alpha\cdot 2^*_s}},
\end{equation}
for some positive constant $C$. Now, in order to apply the Moser's iteration technique  in $\mathbb{R}^N$ and get then the local boundedness of the solution, we take into account that
\begin{equation}\label{eq:estimate1111}
\int_{\mathbb R^N}\frac{v^{2_s^*}}{|x|^{\alpha\cdot 2^*_s}}<+\infty
\end{equation}
and we estimate the right hand side of \eqref{eq:leftright}. Let
\begin{equation}\label{eq:estimate11112}
\beta=\frac{2^*_s}{2}
\end{equation} and let $m\in \mathbb{R}^{+}$ to  be chosen later. We have
\begin{eqnarray}\label{eq:betafix}\\\nonumber
&&C\beta \int_{\mathbb{R}^N} (\phi (v))^2v^{2_s^*-2}\frac{dx}{|x|^{\alpha\cdot 2^*_s}}\\\nonumber&=&
C \beta \int_{\{v\leq m\}\cap\mathbb{R}^N} (\phi (v))^2v^{2_s^*-2}\frac{dx}{|x|^{\alpha\cdot 2^*_s}}+ C\beta\int_{\{v\geq m\}\cap\mathbb{R}^N} \frac{(\phi (v))^2}{|x|^{2\alpha}}\frac{v^{2_s^*-2}}{|x|^{\alpha\cdot (2^*_s-2)}}\,dx\\\nonumber&\leq&C\beta\int_{\{v\leq m\}\cap\mathbb{R}^N} (\phi (v))^2m^{2_s^*-2}\frac{dx}{|x|^{\alpha\cdot 2^*_s}}+C\beta\left( \int_{\{v\geq m\}\cap\mathbb{R}^N} \frac{(\phi (v))^{2^*_s}}{|x|^{\alpha\cdot 2^*_s}}\right)^{\frac{2}{2^*_s}}\left( \int_{\{v\geq m\}\cap\mathbb{R}^N}\frac{v^{2_s^*}}{|x|^{\alpha\cdot 2^*_s}}\right)^{\frac{2^*_s-2}{2^*_s}},
\end{eqnarray}
where in the last term we used H\"older inequality
with exponents~$2^*_s/2$ and~$2^*_s/(2^*_s-2)$.
Since \eqref{eq:estimate1111} holds, we can fix $m$ such that
$$\left( \int_{\{v\geq m\}\cap\mathbb{R}^N}\frac{v^{2_s^*}}{|x|^{\alpha\cdot 2^*_s}}\right)^{\frac{2^*_s-2}{2^*_s}}\leq \frac{1}{2C\beta},$$
and then from \eqref{eq:leftright} and \eqref{eq:betafix}  we obtain
\begin{eqnarray}\nonumber
\left(\int_{\mathbb{R}^N}|\phi(v)|^{2^*_s}\frac{dx}{|x|^{\alpha\cdot 2^*_s}}\right)^{\frac{2}{2^*_s}}&\leq &C \beta m^{2_s^*-2}\int_{\mathbb{R}^N} (\phi (v))^2\frac{dx}{|x|^{\alpha\cdot 2^*_s}}\\\nonumber&\leq&
C\beta  m^{2_s^*-2}\int_{\mathbb{R}^N} v^{2\beta}\frac{dx}{|x|^{\alpha\cdot 2^*_s}}<+\infty,
\end{eqnarray}
where we used that $\phi (t)\leq t^\beta$, \eqref{eq:estimate1111} and \eqref{eq:estimate11112}.  By Fatou's lemma, taking $T\rightarrow +\infty$ one has
\begin{equation}\label{eq:leftright2222}
\left(\int_{\mathbb{R}^N}|v|^{\beta\cdot 2^*_s}\frac{dx}{|x|^{\alpha\cdot 2^*_s}}\right)^{\frac{2}{2^*_s}}<+\infty.
\end{equation}
The result now, follows  using Moser's iteration, see e.g.~\cite[Theorem 13]{LPPS}.  For $k\geq 1$, let  us define $\{\beta_{k}\}$ by
 $$2 \beta_{k+1}+2^*_s-2=2^*_s \beta_k \quad \text{and} \quad \beta_1=\frac{2^*_s}{2}.$$ Then from \eqref{eq:leftright} and  using \eqref{eq:leftright2222}, iterating  we obtain
\begin{equation}\nonumber
\left(\int_{\mathbb{R}^N}|v|^{\beta_{k+1} \cdot 2^*_s}\frac{dx}{|x|^{\alpha\cdot 2^*_s}}\right)^{\frac{1}{2^*_s(\beta_{k+1}-1)}}\leq \big (C \beta_{k+1}\big)^\frac{1}{2(\beta_{k+1}-1)}\left(\int_{\mathbb{R}^N} v^{\beta_k\cdot 2_s^*}\frac{dx}{|x|^{\alpha\cdot 2^*_s}}\right)^\frac{1}{2_s^*(\beta_{k}-1)}.
\end{equation}
If we denote
$$A_k:=\left(\int_{\mathbb{R}^N}|v|^{\beta_{k} \cdot 2^*_s}\frac{dx}{|x|^{\alpha\cdot 2^*_s}}\right)^{\frac{1}{2^*_s(\beta_{k}-1)}}, \qquad C_k:= (C \beta_{k}\big)^\frac{1}{2(\beta_{k}-1)},$$
 we get the recurrence formula $A_{k+1}\leq C_{k+1}A_k$, $k\geq 1$. Arguing by induction we have
 \begin{eqnarray}\label{eq:induction}
 \log A_{k+1}&\leq &\sum_{j=2}^{k+1}\log C_{j}+\log A_1\\\nonumber&\leq&  \sum_{j=2}^{+\infty}\log C_{j}+\log A_1 < +\infty,
 \end{eqnarray}
 since the serie $\displaystyle \sum_{j=2}^{+\infty}\log C_{j}< +\infty$ is convergent (recall that $\beta_{k+1}=\beta_1^{k}(\beta_1-1)+1$) and $A_1\leq C$, see \eqref{eq:leftright2222}.
For $R>0$ fixed, using \eqref{eq:induction}, it follows
$$\log\left (\left(\int_{B_R}|v|^{\beta_{k+1} \cdot 2^*_s}\frac{dx}{|x|^{\alpha\cdot 2^*_s}}\right)^{\frac{1}{2^*_s(\beta_{k+1}-1)}}\right)\leq C$$ and then
$$\frac{\alpha}{(\beta_{k+1}-1)}\log\frac 1R +\log\left(
\left( \int_{B_R}|v|^{\beta_{k+1} \cdot 2^*_s}\,{dx} \right)^{\frac{1}{2^*_s(\beta_{k+1}-1)}}  \right)\leq C.$$
Since $\beta_{k}\rightarrow +\infty$ as $k\rightarrow  \infty$, we have
$$\log \left(\left(\int_{B_R}|v|^{\beta_{k+1} \cdot 2^*_s}\,{dx}\right)^{\frac{1}{2^*_s(\beta_{k+1}-1)}}\right)\leq C,$$
with $C$ a positive constant not depending on $R$.
This end the proof since
$$\lim_{k\rightarrow +\infty}\left(\int_{B_R}|v|^{\beta_{k+1} \cdot 2^*_s}\,{dx}\right)^{\frac{1}{2^*_s\beta_{k+1}}}=\|v\|_{L^{\infty}(B_R)}$$
and thus
$$\|v\|_{L^{\infty}(\mathbb{R}^N)}\leq C,$$
which gives the desired result.
\end{proof}

\subsection{Harnack inequality for $v_\alpha$}\label{sec:harn}

A good reference in order to understand the differences between the Harnack inequality for local and nonlocal operators related to the fractional Laplacian  is the paper~\cite{Kass}. However, such a paper does not apply directly to the operator $(-\Delta_{\alpha})^s$ introduced in~\eqref{operatoralpha},
because of the singularities of the coefficients of the operator.

For our purposes we are going to use
the following weak Harnack inequality, that has been obtained in \cite{AMPP}.

\begin{theorem}\label{thm:peralabdellaoui}
Let~$\alpha\in(0,(N-2s)/2)$.
Let $v\in \dot H^{s,\alpha}(\R^N)$ be a nonnegative solution of
\begin{equation}\label{Eq:Pv-1}
(-\Delta_{\alpha})^s v=\frac{v^{2_s^*-1}}{|x|^{\alpha\cdot 2^*_s}} \quad\text{in}\,\, \mathbb{R}^N.
\end{equation}
Then, for  $1\le q<\frac{N}{N-2s}$ the following inequality holds true
\begin{equation}\label{main}
\Big(\int_{B_r}v^qd\mu(x)\Big)^{\frac 1q}\le C(q,N,s)
\inf_{B_{\frac{3}{2}r}}v, \quad d\mu(x):=
\dfrac{dx}{|x|^{2\alpha}}.
\end{equation}
\end{theorem}

For the readers convenience we describe the strategy of the proof in a  schematic  way and we refer to \cite{AMPP} for the details.
The functional framework needed to prove Theorem~\ref{thm:peralabdellaoui} can be found in Appendix B of \cite{AMPP1},
where a Harnack parabolic inequality is obtained for the heat equation corresponding to the elliptic operator $(-\Delta_{\alpha})^s$.

The proof of Theorem~\ref{thm:peralabdellaoui}
is carried out using classical arguments by Moser and by Krylov-Safonov.

In the local case (that is, when~$s=1$ and the fractional Laplacian
reduces to the Laplacian),
the Harnack inequality for elliptic operator with weights has been proved in \cite{FKS}. In the nonlocal case, we also refer to the paper~\cite{DKP},
in which the authors consider nonlinear operators of  nonlocal $p$-Laplacian type. We also refer to Chapter 7 of the book of Giusti~\cite{Giu}.
\smallskip

For simplicity of notation, we will write $B_r$ in place of $B_r(0)$.
Moreover, we will use the notation
$$d\mu:=\dfrac{dx}{|x|^{2\alpha}} \quad \hbox{ and } \quad d\nu:=
\dfrac{dx\,dy}{|x-y|^{N+2s}|x|^\alpha|y|^\alpha}.$$
The first result toward the proof of the Harnack inequality is contained
in the following lemma, where we check that, even in the presence of singular weights, we get a \textit{propagation of positivity} result.
More precisely:

\begin{lemma}\label{propagation}\textit{(Propagation of positivity)}
Assume that $v\in \dot H^{s,\alpha}(\R^N)$, with $v\gvertneqq 0$ in $B_R(0)$ with $0<R<1$, is a
supersolution to \eqref{Eq:Pv-1}.  Let $k>0$ and suppose that for
some $\sigma\in (0,1]$, we have
\begin{equation}\label{elli1}
|B_r\cap \{v\ge k\}|_{d\mu}\ge \sigma|B_r|_{d\mu}
\end{equation}
with $0<r<\frac{R}{16}$, then there exists a positive constant
$C=C(N,s)$ such that
\begin{equation}\label{elli2}
|B_{6r}\cap \{v\le 2\delta k\}|_{d\mu}\le
\frac{C}{\sigma\log(\frac{1}{2\delta})}|B_{6r}|_{d\mu}
\end{equation}
for all $\delta\in (0,\frac 14)$.
\end{lemma}

We also mention the paper~\cite{FV}, that contains some estimates that
are useful to handle radial weights.

\

An iterative argument as in Lemma 3.2 in \cite{DKP} gives the following local estimate on the infimum of $v$.

\begin{lemma}\label{lema2}
Assume that the hypotheses of Lemma \ref{propagation} are satisfied.
Then there exists $\delta \in (0,\frac 12)$ such that
\begin{equation}\label{estim2} \inf_{B_{4r}} v\ge \delta k.
\end{equation}
\end{lemma}

\

Since the weight $|x|^{-2\alpha}$ is an \textit{admissible weight} in the sense defined in \cite{HKM}, we obtain the following
 \emph{reverse H\"{o}lder inequality}
for $v$.
\begin{lemma}\label{dos} \textit{(Reverse H\"{o}lder inequality)}
Suppose that $v$ is a nonnegative supersolution to \eqref{Eq:Pv-1}, then for all
$0<\gamma_1<\gamma_2<\frac{N}{N-2s}$, we have
\begin{equation}\label{est3}
\left(\dfrac{1}{|B_{r}|_{d\mu}}\int\limits_{B_{r}}v^{\gamma_2}\,d\mu\right)^{\frac{1}{\gamma_2}}
\le C \left(\dfrac{1}{|B_{\frac 32
r}|_{d\mu}}\int\limits_{B_{\frac32
r}}v^{\gamma_1}\,d\mu\right)^{\frac{1}{\gamma_1}}.
\end{equation}
\end{lemma}

\

In order to prove the main result of this section, we also need to estimate
an average of~$v$ by the infimum in a small ball. For this, we state
the following covering lemma in the spirit of Krylov-Safonov theory (see \cite{KS} for a proof in a very general framework).

\begin{lemma}\label{cover}
Assume that $E\subset B_r(x_0)$ is a measurable set. For
$\bar{\delta}\in (0,1)$, we define
$$
\big[E\big]_{\bar{\delta}}:=\bigcup_{\rho>0}\{B_{3\rho}(x)\cap
B_r(x_0), x\in B_r(x_0): |E\cap B_{3\rho}(x)|_{d\mu}>\bar{\delta}|
B_{\rho}(x)|_{d\mu}\}.
$$
Then, there exists $\tilde{C}$ depending only on $N$, such that,  either
\begin{enumerate}
\item $
|\big[E\big]_{\bar{\delta}}|_{d\mu}\ge\dfrac{\tilde{C}}{\bar{\delta}}|E|_{d\mu}$,
or \item $\big[E\big]_{\bar{\delta}}=B_r(x_0)$.
\end{enumerate}
\end{lemma}

With this, we can have the following:

\begin{lemma}\label{tres}\textit{(Main result)}
Assume that $v$ is a nonnegative supersolution to \eqref{Eq:Pv-1}, then there
exists $\eta\in (0,1)$ depending only on $N,s$ such that
\begin{equation}\label{est31}
\dyle\Big(\dfrac{1}{|B_r|_{d\mu}}\int_{B_r}v^{\eta}d\mu(x)\Big)^{\frac{1}{\eta}}\le
C\inf_{B_r} v.
\end{equation}
\end{lemma}

\begin{proof}
For any $\eta>0$ we have,
\begin{equation}\label{rep}
\dyle\dfrac{1}{|B_r|_{d\mu}}\int_{B_r}v^{\eta}d\mu(x)=\eta
\int_0^\infty
t^{\eta-1}\dfrac{|B_r\cap\{v>t\}|_{d\mu}}{|B_r|_{d\mu}}dt.
\end{equation}
For any $t>0$ and $i\in \mathbb{N}$, we set $A^i_t:=\{x\in B_r:
v(x)>t\delta^i\}$, where $\delta $ is given by Lemma \ref{lema2}.
It is easy to check that $A^{i-1}_t\subset A^i_t$.

Let $x\in B_r$ such that $B_{3\rho}(0)\cap B_r\subset
[A^{i-1}_t]_{\bar{\delta}}$, then
$$
|A^{i-1}_t\cap
B_{3\rho}(x)|_{d\mu}>\bar{\delta}|B_\rho|_{d\mu}=\frac{\bar{\delta}}{3^{N-2\g}}|B_{3\rho}|_{d\mu}.
$$
Hence, using Lemma \ref{lema2}, we obtain that
$$
v(x)>\delta (t\delta^{i-1})=t\delta^i  \quad \mbox{  for all }x\in B_r.
$$
Thus $[A^{i-1}_t]_{\bar{\delta}}\subset A^{i}_t$. Therefore, using
the alternatives in Lemma \ref{cover}, we obtain that
\begin{itemize}
\item either $A^{i}_t=B_r$
\item or $|A^{i}_t|_{d\mu}\ge
\dfrac{\tilde{C}}{\delta}|A^{i-1}_t|_{d\mu}$.
\end{itemize}
Hence, if for some $m\in \mathbb{N}$ we have
\begin{equation}\label{nesr}
|A^0_t|_{d\mu}>\left(\dfrac{\bar{\delta}}{\tilde{C}}\right)^m|B_r|_{d\mu},
\end{equation}
then $|A^{m}_t|_{d\mu}= |B_r|_{d\mu}$.  Therefore $A^{i}_t=B_r$ and
then
$$
\inf_{B_r}v>t\delta^m.
$$
It is clear that \eqref{nesr} holds if
$m>\frac{1}{\log\left(\frac{\bar{\delta}}{\tilde{C}}\right)}\log\left(\dfrac{|A^{0}_t|_{d\mu}}{|B_r|_{d\mu}}\right)$.

Fix $m$ as above and   define
$$\beta:=\dfrac{\log\left(\dfrac{\bar{\delta}}{\tilde{C}}\right)}{\log(\delta)}.$$
It follows that
$$
\inf_{B_r}v>t\delta\left(\frac{|A^{0}_t|_{d\mu}}{|B_r|_{d\mu}}\right)^{\frac{1}{\beta}}.
$$
We set $\xi:=\inf_{B_r}v$, then
$$
\dfrac{|B_r\cap\{v>t\}|_{d\mu}}{|B_r|_{d\mu}}=\frac{|A^{0}_t|_{d\mu}}{|B_r|_{d\mu}}\le
C\delta^{-\beta}t^{-\beta}\xi^\beta.
$$
Going back to \eqref{rep}, we have
$$
\dyle\dfrac{1}{|B_r|_{d\mu}}\int_{B_r}v^{\eta}d\mu(x)\le
\eta\int_0^at^{\eta-1}dt +\eta C\int_a^\infty
\delta^{-\beta}t^{-\beta}\xi^\beta dt.
$$
Choosing $a=\xi$ and $\eta=\frac{\beta}{2}$, we obtain the desired result.
\end{proof}

\begin{proof}[Proof of Theorem \ref{thm:peralabdellaoui}]
Using Lemma \ref{tres} we
obtain that
$$
\dyle\left(\dfrac{1}{|B_r|_{d\mu}}\int_{B_r}u^{\eta}d\mu(x)\right)^{\frac{1}{\eta}}\le
C\inf_{B_r} u
$$
for some $\eta\in (0,1)$. Fixed $1\le q<\frac{N}{N-2s}$, and applying
Lemma \ref{dos} with $\gamma_1:=\eta$ and $\gamma_2:=q$, there
results that
\begin{equation}\label{est33}
\left(\frac{1}{|B_{r}|_{d\mu}}\int\limits_{B_{
r}}u^{q}\,d\mu\right)^{\frac{1}{q}} \le C \left(\frac{1}{|B_{\frac
32 r}|_{d\mu}}\int\limits_{B_{\frac32
r}}u^{\eta}\,d\mu\right)^{\frac{1}{\eta}}.
\end{equation}
Hence
$$
\left(\frac{1}{|B_{r}|_{d\mu}}\int\limits_{B_{
r}}u^{q}\,d\mu\right)^{\frac{1}{q}} \le C\inf_{B_{\frac 32 r}} u,
$$
which concludes the proof of Theorem~\ref{thm:peralabdellaoui}.
\end{proof}

\

We refer to \cite{AMPP} for all the technical details of the proofs above and to Appendix B in \cite{AMPP1} for the functional inequalities for weighted fractional Sobolev spaces needed in the proofs of the previous lemmata.

\

\begin{remark}
With the $L^\infty$ estimate and the weak Harnack inequality we could obtain  the classical Harnack inequality. We omit the details because they are quite classical.
\end{remark}

\subsection{Proof of Theorem \ref{thm:main}.}
We start studying  the behavior of the solution $u$ near the origin.
Defining
$$v_{\alpha_{\theta}}(x)=|x|^{\alpha_{\theta}}u(x),$$ for $R_0>0$,  by the  Harnack inequality, see Theorem \ref{thm:peralabdellaoui},  we get that
\begin{equation}\label{eq:asmp1}
C_H\inf_{B_{R_0}}v_{\alpha_{\theta}}(x)\geq
\left(\int_{B_{R_0}}v^q\,dx\right)^{\frac 1q}\geq c_0,
\end{equation}
for some positive constant $c_0$.
On the other hand, Proposition~\ref{pro: linfty} implies
the existence of a positive constant $C_0$ such that
\begin{equation}\label{eq:asmp2}
v_{\alpha_{\theta}}(x) \leq C_0, \quad x\in B_{R_0}.
\end{equation}
Then,  from \eqref{eq:asmp1} and \eqref{eq:asmp2} (recalling that $v_{\alpha_{\theta}}(x)=|x|^{\alpha_{\theta}}u(x)$),   it  follows
\begin{equation}\nonumber
\frac{c_0}{|x|^{\alpha_{\theta}}}\leq u(x) \leq \frac{C_0}{|x|^{\alpha_{\theta}}}
\quad \text{in}\,\,B_{R_0}.
\end{equation}
Thus, recalling~\eqref{4.16bis},
\begin{equation}\label{eq:stima1}
\frac{c_0}{|x|^{(1-\eta_{\theta})\frac{N-2s}{2}}}\leq u(x)
\leq \frac{C_0}{|x|^{(1-\eta_{\theta})\frac{N-2s}{2}}} \quad \text{in}\,\,B_{R_0}.
\end{equation}

In order to study the behavior of $u(x)$ at  infinity,
we use the Fractional Kelvin transform, see e.g. \cite[Proposition A.1]{quim}.
Let $x\rightarrow x^*={x}/{|x|^2}$ the inversion with respect to the unit sphere
and let us define
\begin{equation}\label{eq:ukelvin}
u^*(x):=|x|^{2s-N}u(x^*).
\end{equation}
Then, from~\cite[Proposition A.1]{quim} we have that
\begin{equation}\label{eq:ashgadlhasghg}(-\Delta)^su^*(x)
=\frac{1}{|x|^{N+2s}}(-\Delta)^su(x^*), \quad x\neq0.
\end{equation}
Using~\eqref{Eq:P} and~\eqref{eq:ashgadlhasghg},
formally we obtain that
\begin{eqnarray*}
(-\Delta)^su^*(x)&=&
\frac{1}{|x|^{N+2s}} \left[\theta \frac{u\left(\frac{x}{|x|^2}\right)}{\left|\frac{x}{|x|^2}\right|^{2s}}+u^{2^*_s-1}\left(\frac{x}{|x|^2}\right)\right]\\
&=&\theta\frac{u^*(x)}{|x|^{2s}}+{\big(u^*\big)}^{2^*_s-1},  \quad x\neq0.
\end{eqnarray*}
Moreover, from \cite[Lemma 2.2 and Corollary 2.3]{FW} we have that $u^*\in \dot{H}^s(\mathbb{R}^N)$ and it is a weak solution of the problem
\begin{equation}\label{eq:u^*}
(-\Delta)^su^*(x)
=\theta\frac{u^*(x)}{|x|^{2s}}+{\big(u^*\big)}^{2^*_s-1} \quad {\mbox{ in }}
\R^N\setminus\{0\}.
\end{equation}
Then arguing as in the first part of the proof, for  a fixed ${1}/{R_{\infty}}>0$,  there exist two positive constants  $c_{\infty}$ and $C_{\infty}$ such that
\begin{equation}\label{eq:stima2}
\frac{c_{\infty}}{|x|^{\alpha_{\theta}}}\leq u^*(x)
\leq \frac{C_{\infty}}{|x|^{\alpha_{\theta}}} \quad \text{in}\,\,B_{\frac{1}{R_{\infty}}}.
\end{equation}
Scaling back in \eqref{eq:stima2}, see \eqref{eq:ukelvin}, we obtain
\begin{equation}\label{eq:stima22}
\frac{c_{\infty}}{|x|^{(1+\eta_{\theta})\frac{N-2s}{2}}}\leq u(x) \leq
\frac{C_{\infty}}{|x|^{(1+\eta_{\theta})\frac{N-2s}{2}}}
\quad \text{in}\,\,\mathbb{R}^N\setminus B_{R_{\infty}}.
\end{equation}
Redefining  constants, from \eqref{eq:stima1} and
\eqref{eq:stima22}, we get \eqref{eq:sdafdsfa},
and this concludes the proof of Theorem~\ref{thm:peralabdellaoui}.

\end{document}